\documentclass[11pt,reqno]{amsart}
\usepackage{amssymb,latexsym}
\usepackage{graphicx}
\usepackage{url}		

\calclayout

\makeatletter
\newcommand{\monthyear}[1]{%
	\def\@monthyear{\uppercase{#1}}}
\newcommand{\volnumber}[1]{%
	\def\@volnumber{\uppercase{#1}}}
\AtBeginDocument{%
	\def\ps@plain{\ps@empty
		\def\@oddfoot{\@monthyear \hfil \thepage}%
		\def\@evenfoot{\thepage \hfil \@volnumber}}
	\def\ps@firstpage{\ps@plain}
	\def\ps@headings{\ps@empty
		\def\@evenhead{%
			\setTrue{runhead}%
			\def\thanks{\protect\thanks@warning}%
			\uppercase{}\hfil}%
		\def\@oddhead{%
			\setTrue{runhead}%
			\def\thanks{\protect\thanks@warning}%
			\hfill\uppercase{Generalized Fibonacci spiral}}%
		\let\@mkboth\markboth
		\def\@evenfoot{%
			\thepage \hfil \@volnumber}%
		\def\@oddfoot{%
			\@monthyear \hfil \thepage}%
	}%
	\footskip=25pt
	\pagestyle{headings}%
}
\makeatother

\newcommand{\R}{{\mathbb R}}

\newcommand{\N}{{\mathbb N}}
\newcommand{\Z}{{\mathbb Z}}

\theoremstyle{plain}
\numberwithin{equation}{section}
\newtheorem{thm}{Theorem}[section]

\newtheorem{proposition}[thm]{Proposition}

\begin{document}
	\monthyear{}
	\volnumber{}
	\setcounter{page}{1}
	
	\title{A generalized Fibonacci spiral}
	\author{Bernhard R. Parodi}
	\address{Gewerblich-industrielles Bildungszentrum Zug, Baarerstrasse 100, CH-6301 Zug, Switzerland}
	\email{bernhard.parodi@gibz.ch}
	\thanks{...}
	
\date{20 April 2020}

\begin{abstract}
	As a generalization of planar Fibonacci spirals that are based on the recurrence relation $F_n=F_{n-1}+F_{n-2}$, we draw assembled spirals stemming from analytic solutions of the recurrence relation $G_n=a\, G_{n-1}+b\, G_{n-2}+c\, d\,^n$, with positive real initial values $G_0$ and $G_1$ and coefficients $a$, $b$, $c$, and $d$. The principal coordinates given in closed-form correspond to finite sums of alternating even- or alternating odd-indexed terms $G_{n}$. For rectangular spirals made of straight line segments (a.k.a. spirangles), the even-indexed and the odd-indexed directional corner points asymptotically lie on mutually orthogonal oblique lines. We calculate the points of intersection and show them in the case of inwinding spirals to coincide with the point of convergence. In the case of outwinding spirals, an $n$-dependent quadruple of points of intersection may form. For arched spirals, interpolation between principal coordinates is performed by means of arcs of quarter-ellipses. A three-dimensional representation is exhibited, too. The continuation of the discrete sequence $\{G_n\}$ to the complex-valued function $G(t)$ with real argument $t$$\in$$\R$, exhibiting spiral graphs and oscillating curves in the Gaussian plane, subsumes the values $G_n$ for $t$$\in$$\N$ as the zeros.  Besides, we provide a matrix representation of $G_n$ in terms of transformed Horadam numbers, retrieve the Shannon product difference identity as applied to $G_n$, and suggest a substitution method for finding a variety of other identities and summations related to $G_n$.
\end{abstract}

\maketitle
\section{Introduction}

Spiral representations of the (integer valued) Fibonacci numbers and of some of their generalizations (e.g., \cite{oeis}) are known for long and applied both naturally and artificially in two (e.g., \cite{davis71}, \cite{holden75}, \cite{hoggatt76}, \cite{engstrom87}, \cite{ozvatan}, \cite{koshy}) as well as in three dimensions (e.g., \cite{harary11},\cite{nagy18}). For a non-exhausting review on applications of the Fibonacci numbers with some more occurences concerning spirals, see \cite{pletser}. The underlying number generation principle is equivalent to a deterministic nonhomogeneous linear second-order recurrence relation as treated in most introductory textbooks on difference equations (e.g., \cite{jeske63} or recently \cite{elaydi05}). Investigations with the genuine Fibonacci sequence and with some of its various generalizations are often involving an explicit Moivre-Binet-form solution. This is particularly encountered with Horadam sequences without \cite{buschman}, \cite{horadam65} and with an additive constant input \cite{bicknell}, \cite{zhang97}, \cite{filiponi}, \cite{phadte16}.   The explicit solution for the generalized case with real constant coefficients and with an additive exponential input has been recently discussed by Phadte \& Pethe (2013) in \cite{phadte13} and Phadte \& Valaulikar (2016) in \cite{phadte16}. Here we extend their work, firstly, by a couple of finite sums that are applied to, secondly, a visualization of such generalized Fibonacci numbers by means of assembled spirals. Similar to the generation of the original Fibonacci spiral, our piecewise construction uses parametric plots for the interpolation between principal coordinates of each segment. We focus on rectangular spirals (a.k.a. spirangles) and on arched spirals, both in two and three dimensions. Our graphical representations are restricted to lines, we won't draw corresponding spiraling surfaces in 3D space here. Similar to the treatment of Horadam \& Shannon (1988) \cite{horadam88} and Horadam (1988) \cite{horadam88b} for Fibonacci, Lucas, Pell, and Jacobsthal curves, or recently by, e.g., Chandra \& Weisstein \cite{chandra} and \"{O}zvatan \& Pashaev (2017) \cite{ozvatan} for Fibonacci numbers, we exhibit the exponentially generalized Fibonacci sequence by means of a continuous parametric plot in the complex plane. This produces, dependent on the parameter values chosen, either oscillatory or spiraling graphs.

The content of the paper is as follows. In Section 2 the explicit (non-degenerate) solution for the generalized Fibonacci sequence with exponential input is derived, a related generalized product difference identity of order is formulated and shown to correspond to a generalized Shannon identity, and the finite sums of alternating even- or odd-indexed terms are given in closed form for proper use in the sequel. In Section 3, following a short ad-hoc classification of spirals and relying on the results of the previous section, the formalism proposed to draw rectangular and arched two- and three-dimensional generalized Fibonacci spirals is presented and graphical examples are shown. Specific features like the orthogonal positioning of directional corner points in outwinding spirals or some point of convergence for inwinding spirals are discussed and calculated. In Section 4 the index $n$ is replaced by some real variable $t$, leading to generalized Fibonacci spirals and curves in the Gaussian plane. The conclusions are drawn in Section 5, followed by the Appendix in Section 6. The latter harbours a matrix representation for the generalized Fibonacci sequence and a related decomposition into fundamental Horadam numbers as well as some summation formulae, including the principal proof of the paper.    

\section{Generalized Fibonacci sequence with an exponential input} 

The deterministic nonhomogenous linear second-order difference equation
\begin{equation} \label{RR}
G_n=a\, G_{n-1}+b\, G_{n-2}+c\, d\,^n,\,\,\,\,\,n\ge 2  
\end{equation}
with real initial values $G_0, G_1 \in \mathbb{R}$ and subject to real constant coefficients $a$, $b$ for the homogeneous part and alike $c$, $d$ for the inhomogeneous input, has a well-known solution, as recast below for the sake of completeness. If the inhomogeneous input is written as $c\,d^{\,n-m}$, $m\in \mathbb{Z}$, all the following results stay valid with a correspondingly substituted value $c\,d^{-m}$ instead of our choice $c$. Actually, in \cite{phadte13} and \cite{phadte16} the value $m=2$ is used. The sequence corresponds to a \emph{generalized Horadam sequence} \cite{horadam65},\cite{larcombe13}, with the generalization being due to the exponential input and due to the allowance of real parameter values (see below). In the same sense, one may call it a \emph{generalized bivariate Fibonacci sequence} or a generalized Fibonacci polynomial in $a$ and $b$ (e.g., \cite{catalani}) \emph{with an exponential input}.  For the sake of simplicity, we restrict our inquiry hereafter to parameter values $a$, $b$ $\in\R^+$ and $c$, $d$ $\in\R_0^+$.

The formal \textit{extension to negative indices} $-n<0$ is found by means of writing relation \eqref{RR} as $G_{-n+2}=aG_{-n+1}+bG_{-n} +cd^{-n+2}$, yielding 
\begin{equation}\label{RRneg}
G_{-n}=\frac{1}{b} \Big\{G_{-n+2}-aG_{-n+1}-cd^{-n+2}\Big\}.
\end{equation}  
In particular, applying this twice provides for later use
\begin{eqnarray}\label{negind}
G_{-1} & = & \frac{1}{b} \Big\{ G_1-a\,G_0-c\,d \Big\}\\
G_{-2} & = & \frac{1}{b^2} \Big\{ (a^2+b)\,G_0-a\,G_1+c\,(a\,d-b)\Big\}. 
\end{eqnarray}
For practical purposes, any negative integer index value may simply be plug into the formula for the general solution that will be derived in the next section. 

\subsection{Explicit solution} 

Inserting the Ansatz $G_n = A\,\lambda^n$ into the corresponding homogeneous (or reduced) recurrence relation \eqref{RR} (i.e., with $c=0$) yields the characteristic equation 
\begin{equation}
\lambda^2-a\lambda-b = 0, \label{chequ}
\end{equation}
with the fundamental set of solutions $\alpha := \lambda_1$, $\beta := \lambda_2$ given by the quadratic formula, i.e., 
\begin{equation}
\alpha=\frac{a+ \sqrt{a^2+4b}}{2} \hspace{1cm}
\beta=\frac{a- \sqrt{a^2+4b}}{2} .
\label{lambda12}
\end{equation}
We restrict our objects of interest to real and distinct solutions $\alpha, \beta \in \mathbb{R}$, $\alpha \ne \beta $, implying a positive discriminant $a^2+4b>0$ (henceforth called restriction 1). Further solutions either for the case $\alpha=\beta=a/2$ or for oscillatory solutions $\alpha, \beta \in \mathbb{C}$ are not tackled here. The roots obey the relations 
\begin{equation}
	\alpha+\beta=a, \hspace{0.5cm}\alpha\beta=-b, 
	\label{alphabeta1}
\end{equation}
\begin{equation}
 (1-\alpha)(1-\beta)=1-a-b, \hspace{0.5cm}
 (d-\alpha)(d-\beta)=d^2-ad-b,
	\label{alphabeta2}
\end{equation}
\begin{equation}
	\alpha^2=\alpha\,a +b,\hspace{0.5cm}\beta^2=\beta\,a +b,
	\hspace{0.5cm} \alpha^2+\beta^2 = a^2+2b, \label{alphabeta3}
\end{equation}
\begin{equation}
	(\alpha^2+1)(\beta^2+1)\,=\,a^2+ (b+1)^2.\label{alphabeta4}
\end{equation}
Hereby the first couple of relations is identical to Vieta's theorem on quadratic equations. As elaborated in many introductory textbooks on difference equations (usually by means of a single worked-through numerical example), the general (or complete) solution of the exponentially nonhomogeneous difference equation \eqref{RR}, subject to four constants and two initial values, is given by
\begin{equation}\label{Gngen}
G_n \,\,=\,\,G_n(a,b,c,d; G_0,G_1)\,\,\,=\,\,\, G_n^{(h)}+G_n^{(p)},
\end{equation}
with integer indices $n\in \Z$, a homogeneous solution (for $c$=0)
\begin{equation}\label{Gnhomo}
G_n^{(h)}=
A\,\alpha^n+B\,\beta^n \hspace{1.6cm}  
\end{equation}
and a particular solution (for $c$$\ne$0, found by inserting a trial term $G_n^{(p)}=pd^n$ into the recurrence relation)
\begin{equation}\label{Gnpart}
G_n^{(p)}=\frac{c\,d^2}{d^2-a\,d-b}\,d^{n} \,=: p\, d^{n}
\hspace{0.6cm} (d^2-a\,d-b \ne 0). 
\end{equation}
We note that our inquiry is restricted to parameter values satisfying the condition $d^2-a\,d-b = (d-\alpha)(d-\beta)\ne 0$ (restriction 2). Degenerate cases with $d=\alpha$ or $d=\beta$ can be dealt with in a follow-up study. 

The constants $A$ and $B$ are formally specified with respect to the initial values $G_0$ and $G_1$: the corresponding general solutions $G_n=A\,\alpha^n+B\,\beta^n+p\,d^{n}$ with $n$=0 or $n$=1 constitute a linear system of equations for $A$ and $B$, namely $G_0 = A+B+p$ and $G_1=A\alpha+B\beta+p\,d$, yielding 
\begin{eqnarray}
A &=& \,\,\,\,\,\,\,\,\frac{G_1-p\,d-(G_0-p)\beta}{\alpha-\beta} \,\,\,\,\,\, =
\,\,\,\,\,\,\,\,\,\frac{H_1 -H_0\beta}{\alpha-\beta}  
\label{A311}\\
B &=& \,\,\,-\,\frac{G_1-p\,d-(G_0-p)\alpha}{\alpha-\beta}\,\,\,\,\, = \,\,\,\, -\,\frac{H_1 -H_0\alpha}{\alpha-\beta},
\label{B311}
\end{eqnarray}
with $\alpha-\beta=\sqrt{a^2+4b}>0$ (due to restriction 1). The transformed numbers
\begin{equation}\label{Hn} 
H_n\equiv G_{n}-p\,d^{\,n} = G_n(a,b,0,0;H_0,H_1)
\end{equation}
are but (real valued) Horadam numbers, obeying the recursion relation 
\begin{equation}
H_n = a\,H_{n-1}+b\,H_{n-2},
\end{equation}
with initial values $H_0$, $H_1$ and with the explicit solution given by equation \eqref{Gnhomo}, together with constants $A$ and $B$ as above. Actually, in equation \eqref{Gnhomo}, $G_n^{(h)}=H_n$. For a matrix representation of $G_n$ in terms of $H_n$, and for a related decomposition, we refer to the Appendix.

To summarize, equations \eqref{Gngen}-\eqref{B311} constitute a fully specified explicit solution
\begin{equation}\label{Gn}
G_n=A\,\alpha^n+B\,\beta^n+p\,d^{n} =H_n+\,p\,d^{n}
\end{equation} 
for the implicit relation \eqref{RR}. For the special value $d$=1 one relies upon the solution of a nonhomogenous second-order recurrence relation with constant coefficients. We finally note that on defining $\gamma :=\max\{\lvert\alpha\rvert,\lvert\beta\rvert,\lvert d\rvert\}$ for the dominant characteristic root or base  (with distinct $\alpha$, $\beta$, and $d$), the dominant solution in the very large-$n$-limit behaves as
\begin{equation}
n\gg 1:\hspace{1cm}G_n \,\,\,\begin{cases}
\propto \gamma^n\hspace{0.7cm}(\gamma \ge 1)\\
\approx 0 \hspace{0.95cm}(\gamma < 1)
\end{cases}.\label{Gnlarge}
\end{equation}

\subsection{A selected identity}

Selected sums and identities involving $G_n$ are given by Phadte \& Pethe (2013) in \cite{phadte13} and by Phadte \& Valaulikar (2016) in \cite{phadte16}. A couple of their results with respect to summations are reformulated in our notation in the Appendix. Here we extend their work with an additional nonlinear identity and in Section 2.3 with a novel summation formula that will be applied in Section 3. By insertion of the solution formula for $G_n$ (equ.  \ref{Gn}) it is a straightforward algebraic exercise to obtain the following \emph{product difference identity}
\begin{equation}
G_{n+u}G_{n+v} -G_{n+u+v}G_{n} =  (-1)^{n+1}b^n\,AB\,(\alpha^u-\beta^u)(\alpha^v-\beta^v)
\nonumber
\end{equation}
\begin{equation}\label{identity1a}
\hspace{0.8cm} +p\,d^{\,n+u+v}
\bigg\{ G_{n+u}d^{-u} +G_{n+v}d^{-v}-G_{n+u+v}d^{-u-v}-G_{n}\bigg\}.
\end{equation}
Inserting the conjoined \textit{fundamental} Horadam numbers $h_n=G_n(a,b,0,0;0,1)=(\alpha^{n}-\beta^{n})/ (\alpha-\beta)$ (i.e., $p=0$ and with initial values 0 and 1, corresponding to the Lucas sequence of the first kind), the first term on the right-hand side may be reexpressed as
\begin{equation}\label{identity1a2}
  (-1)^{n+1}b^n\,AB\,(\alpha^u-\beta^u)(\alpha^v-\beta^v)
 = (-b)^n(H_1-H_0\alpha)(H_1-H_0\beta)\,h_uh_v.
\end{equation}
For the special case of Horadam numbers $H_n=G_n(a,b,0,0;H_0,H_1)$ (i.e., with $p=0$), identity \eqref{identity1a} was already presented together with relation \eqref{identity1a2} in an equivalent form by Horadam (1987) \cite{horadam87a}, \cite{horadam87b}. Alternatively, adopting the notation in terms of transformed Horadam numbers (equation \ref{Hn}), relation \eqref{identity1a} may succinctly be written as
\begin{equation}\label{identity1b}
H_{n+u} H_{n+v}-H_{n+u+v} H_{n}=(-b)^n\, \bigg\{ \,H_u\,H_v-H_{u+v}H_0\bigg\}.
\end{equation}
This elegant product difference identity of order two for generalized Fibonacci numbers seems to have been formulated first by Shannon (1988) in \cite{shannon88} (Lemma 2.3) for Horadam numbers and henceforth is termed \emph{Shannon identity}. It obviously constitutes a generalisation of the well-known Tagiuri identity $F_{n+u}F_{n+v} -F_{n+u+v}F_{n}=(-1)^n\,F_u\,F_v$ for Fibonacci numbers $F_n=G_n(1,1,0,0;0,1)$ (\cite{tagiuri01},\cite{everman},\cite{vajda}). For properly chosen indices $u$ and $v$, relation \eqref{identity1b} subsumes generalized versions of, for example, the Catalan identity ($u=-v$) and the d'Ocagne identity ($u=m-n$, $v=1$) (e.g., \cite{knott}). Here we even allow for real valued and transformed numbers $H_n=G_n-pd^n$. In a reversed presentation of a proof, relation \eqref{identity1a} for the generalized Fibonacci numbers follows from Shannon's identity for Horadam numbers by simply substituting the transformation \eqref{Hn}. Actually, \emph{all known identities for (pure) Horadam numbers $H_n$ must hold for the substitution $G_n-pd^n$ as well, thus indirectly providing a variety of available identities and summations for the generalized Fibonacci numbers $G_n$}. For example, the higher-order identity for $G_n^3-b^3G_{n-1}^3$ or the weighted sum $\sum_{k=1}^{n-1}(-b)^kG_n$ could be found by means of the corresponding identity and sum formulae for Horadam numbers, given in \cite{horadam65} as equations (4.25) and (4.26), respectively. An actual application of this substitution method is provided in Section 6.3.2.

\subsection{Sums of alternating even- or odd-indexed terms} 

Closed-form expressions for the sums $\sum_{k=0}^{n}\,G_{k}$ and $\sum_{k=0}^{n} (-1)^{k}\,G_{k}$ are given in the Appendix for later use. The sums $\sum_{k=0}^{n}\,G_{2k}$ and $\sum_{k=0}^{n}\,G_{2k+1}$ could easily be deduced by means of the substitution method suggested above upon using the corresponding results for Horadam numbers (\cite{horadam65}, equations 3.12 and 3.13). Additionally, in order to prepare some visuali\-zation of the sequence $\{G_n\}$ by means of rectangular and curvy spirals, another couple of sums of differences is calculated here. 

\subsubsection{Formalism.} We are concerned with the sum of the first \emph{alternating even-indexed} terms up to term $G_n$ ($n$ even), $\sum_{k=0}^{n/2} (-1)^{k}\,G_{2k}$, and with the sum of the first \emph{alternating odd-indexed} terms up to term $G_n$ ($n$ odd), $\sum_{k=0}^{(n-1)/2} (-1)^{k}\,G_{2k+1}$. Written compactly, we seek
\begin{eqnarray}\label{SumAlt} 
\Gamma_{n}^{}\,\equiv\,
\sum_{k=0}^{(n-\nu)/2} (-1)^{k}\,G_{2k+\nu}\,,\hspace{0.2cm} \,\hspace{0.4cm}\nu = n\,\mathtt{mod}\,2\,\in\,\{0,\,1\}.
\end{eqnarray}
The mutual summation function $\Gamma_{n}^{}$ ---with index variable $n$ and with an index-dependent and thus \textit{constrained} binary parameter $\nu$ that rules the case distinction between even-indexed terms ($\nu$=0) and odd-indexed terms ($\nu$=1)--- can be expressed in closed form as  
\begin{eqnarray}\label{Gamman}
\Gamma_{n}^{} &=& \frac{1}{a^2+(b+1)^2} \bigg\{ (-1)^{(n-\nu)/2}\Big( G_{n+2}+b^2 G_n \Big) +G_\nu +b^2G_{\nu-2}
\nonumber\\ 
& &-\,c\,\frac{d^2+ad-b}{d^2+1} \,
\bigg( \,(-1)^{(n-\nu)/2}d^{n+2}+d^{\nu}\, \bigg)\,\bigg\},
\,\,\,\,\,\,\nu = n\,\mathtt{mod}\,2,  
\end{eqnarray}
where $G_{-1}^{}$ and $G_{-2}$ are given by equations \eqref{negind}f or simply are $G_n$ evaluated for $n=-1$ and $n=-2$, respectively.  A couple of proofs is relegated to the Appendix.

Consequently, the simple difference relation  
\begin{equation}\label{diffrel}
\Gamma_{n}-\Gamma_{n-2}=(-1)^{(n-\nu)/2}G_n
\end{equation}
holds, as can easily be checked by means of either definition \eqref{SumAlt} or expression \eqref{Gamman}. In the very large-$n$-limit, one finds due to relation \eqref{Gnlarge}  
\begin{equation}\label{Gammanlarge}
n\gg 1:\hspace{1cm}\Gamma_n \,\,\,
\begin{cases}\propto(-1)^{(n-\nu)/2}\,\,\gamma^{n+2} \hspace{5.0cm}(\gamma \ge 1)\\
\approx\frac{1}{a^2+(b+1)^2} \bigg\{ G_\nu +b^2G_{\nu-2}
-\,c\,\frac{d^2+ad-b}{d^2+1}\,d^{\nu}\,\,\bigg\}
\hspace{1cm}(\gamma < 1).
\end{cases}
\end{equation}

We note that the ($\gamma\ge 1$)-case is divergent, while the ($\gamma< 1$)-case converges to a constant (albeit $\nu$-dependent) value. 

\subsubsection{Examples.} Some special cases for $\Gamma_n$ are given now by means of providing algebraic and numerical examples, hereby partly relying upon integer values for $G_n$. Firstly, the sum of the \emph{alternating even- or odd-indexed Fibonacci numbers} $F_n=G_n(1,1,0,0;0,1)$ is calculated with
\begin{eqnarray}
\Gamma_{n}^{\mathtt{Fibo}} &=& \frac{1}{5} \bigg\{ \,\,(-1)^{(n-\nu)/2}\big( F_{n+2}+F_n \big)  +F_\nu+F_{\nu-2} \,\bigg\}, 
\end{eqnarray} 
where $F_{-1}=F_1=1$ and $F_{-2}=-F_2=-1$ according to the negation rule $F_{-n}=(-1)^{n+1}F_n$. For example, taking $n=9$ (hence, $\nu=1$) one has $\sum_{k=0}^{4}\,(-1)^k F_{2k+1}=F_1-F_3+F_5-F_7+F_9=1-2+5-13+34 = 25$ that is equal to $\Gamma_{9}^{\mathtt{Fibo}}=(1/5)(F_{11}+F_9+F_1+F_{-1})=(1/5)(89+34+1+1)=25$. Hansen (1978) \cite{hansen} already  provided some general summation formula for alternating even- or odd-indexed Fibonacci numbers.
  
Secondly, the sum of the first \emph{alternating even- or odd-indexed Horadam numbers} $H_n=G_n(a,b,0,0;H_0,H_1)$ (including the Lucas, the Pell, or some other famous sequences) can be calculated by means of
\begin{eqnarray}\label{GammaHora}
\Gamma_{n}^{\mathtt{Hora}} &=& \frac{1}{a^2+(b+1)^2} \bigg\{ \,\,(-1)^{(n-\nu)/2} \Big( H_{n+2}+\,b^2H_{n}  \Big) +\,H_{\nu}+b^2\,H_{\nu-2}\,\bigg\},
\end{eqnarray} 
with $H_{-1}=(H_1-aH_0)/b$ and $H_{-2}=((a^2+b)H_0-aH_1)/b^2$ according to the negation rules \eqref{negind}f. For example, specifying $a$=1 and $b$=1 and renaming the two initial values will reproduce the particular sum formulas given in \cite{walton} by Walton \& Horadam (1974, equations 4.11f). We note that knowing relation \eqref{GammaHora} in advance (as was not the case with respect to the author), deriving the summation formula \eqref{Gamman} is basically a matter of substitution only. In order to exemplify the substitution method, we correspondingly perform a second proof of \eqref{Gamman} in the Appendix.

To summarize this section: Having put the two restrictions $a^2$$+$$4b$$>$$0$ and $d^2$$-$$a\,d$$-$$b$$\ne$$0$, identity \eqref{identity1a} and \eqref{identity1b} and series \eqref{SumAlt} together with summation function \eqref{Gamman} hold for all integer as well as all real solutions \eqref{Gngen}-\eqref{B311} of the defining relation \eqref{RR}.

\section{Generalized Fibonacci spirals}

\subsection{Classification of spirals}
Based on the generalized sequence discussed in the previous section, the following twofold ad-hoc classification is provisionally suggested for the variety of corresponding spiral representations:
 
I. Algebraic structure. Refering to frequently heard modes of expression and based on the second-order recurrence relation, we may distin\-guish four types of sequences and corresponding spirals according to a growing degree of generalization: 
\begin{itemize}
\item[(i)] the initial or genuine \textit{Fibonacci numbers}  $G_n(1,1,0,0; 0,1)$ constitute a corresponding ordinary or \textit{genuine Fibonacci spiral}.
\item[(ii)] The \textit{generalized genuine}  Fibonacci numbers $G_n(1,1,0,0; G_0,G_1)$ with any real valued pair of initial values $G_0$, $G_1$. This results in \textit{generalized genuine Fibonacci spirals}. 
\item[(iii)] The \textit{generalized bivariate} Fibonacci numbers $G_n(a,b,0,0; G_0,G_1)$ with real coefficients $a$, $b$ and real initial values $G_0$, $G_1$. In the context of integer coefficients and integer initial values, these are often called \textit{Horadam numbers}. The numbers $G_n(a,b,0,0; 0,1)$  original/genuine/fundamental Horadam numbers. Here we allow for real values in general (cf. Section 2.2).  Accordingly, one gets \textit{bivariate Fibonacci spirals} or \emph{Horadam spirals}.
\item[(iv)] The generalized Fibonacci numbers $G_n(a,b,c,d; G_0,G_1)$ with an \textit{exponential input}. For brevity this could be termed "exponacci numbers". The corresponding spirals may then be called \textit{exponentially generalized Fibonacci spirals} or \textit{exponacci spirals}. 
\end{itemize}
For short, cases (ii) to (iv) may all be addressed as \textit{generalized Fibonacci sequences}, irrespective of being either integer or real sequences. Correspondingly, there are then three types of \textit{generalized Fibonacci spirals}. 

II. Geometric composition. A second class of classification concerns the geometric construction of the spirals and hence their visual appearance. As will be evident, we distinguish 
\begin{itemize}
\item[(i)] \textit{shapes} between \textit{rectangular}, \textit{angular}, and \textit{arched} spirals, all of which are composed of elementary graphs added together. While rectangular and angular spirals are drawn by connecting successive points with straight lines, arched spirals are composed of assembled quarter-circles (relying on cases I.i and I.ii from above) or with quarter-ellipses (based on cases I.iii and I.iv). Another type of curved spiral is a spiral given in \textit{polar form}; it is not based on a discret sequence but on a continuous function and will be treated in section 4. A further geometric characteristics involves 
\item[(ii)] \textit{winding}, whereby outwinding means spiraling outwards and inwinding means spiraling inwards. The formal criterion for this distinction is due to the characteristic value $\gamma :=\max\{\lvert\alpha\rvert,\lvert\beta\rvert,\lvert d\rvert\}$ introduced above, 
\begin{equation}
\gamma \ge 1:\,\,\,\mathtt{outwinding},\,\,\,\,\,\,\,\,\,
\gamma < 1:\,\,\,\mathtt{inwinding}.
\end{equation}
For the first case, some parameter values may allow for some inward spiraling for small values of $n$ (as long as $A\alpha^n+B\beta^n$, with $\alpha<1$ and $\beta<1$, dominates the values for $G_n$), followed by outward spiraling for $n$ large enough (as soon as $pd^n$, with $d>1$, dominates $G_n$). We note that if $\gamma=1$, the spiral approaches an attractor of quadratic shape and exhibits cyclic behaviour. The second case ($\gamma<1$)  corresponds to the combined condition $d<1$ and $0<b<1-a^2/2$ (as follows from relation \ref{alphabeta3}, requiring $\alpha^2<1$, $\beta^2<1$). This corresponds to a stable recurrence with the iterates converging to a fixed value (see Section 3.2.2).
\item[(iii)] Last but not least there's leftward or rightward \textit{orientation}. Typically, negative parameter values $a$ or $b$ may govern rightward orientation.
\end{itemize}
For example, one may have a leftwardly outwinding rectangular spiral (as in the left panel of Figure 1, black line) or a leftwardly inwinding arched spiral (right panel, red line). \newline 

\begin{figure*}	\label{Fig1}
	\centering 
	\includegraphics[width=0.49\textwidth]{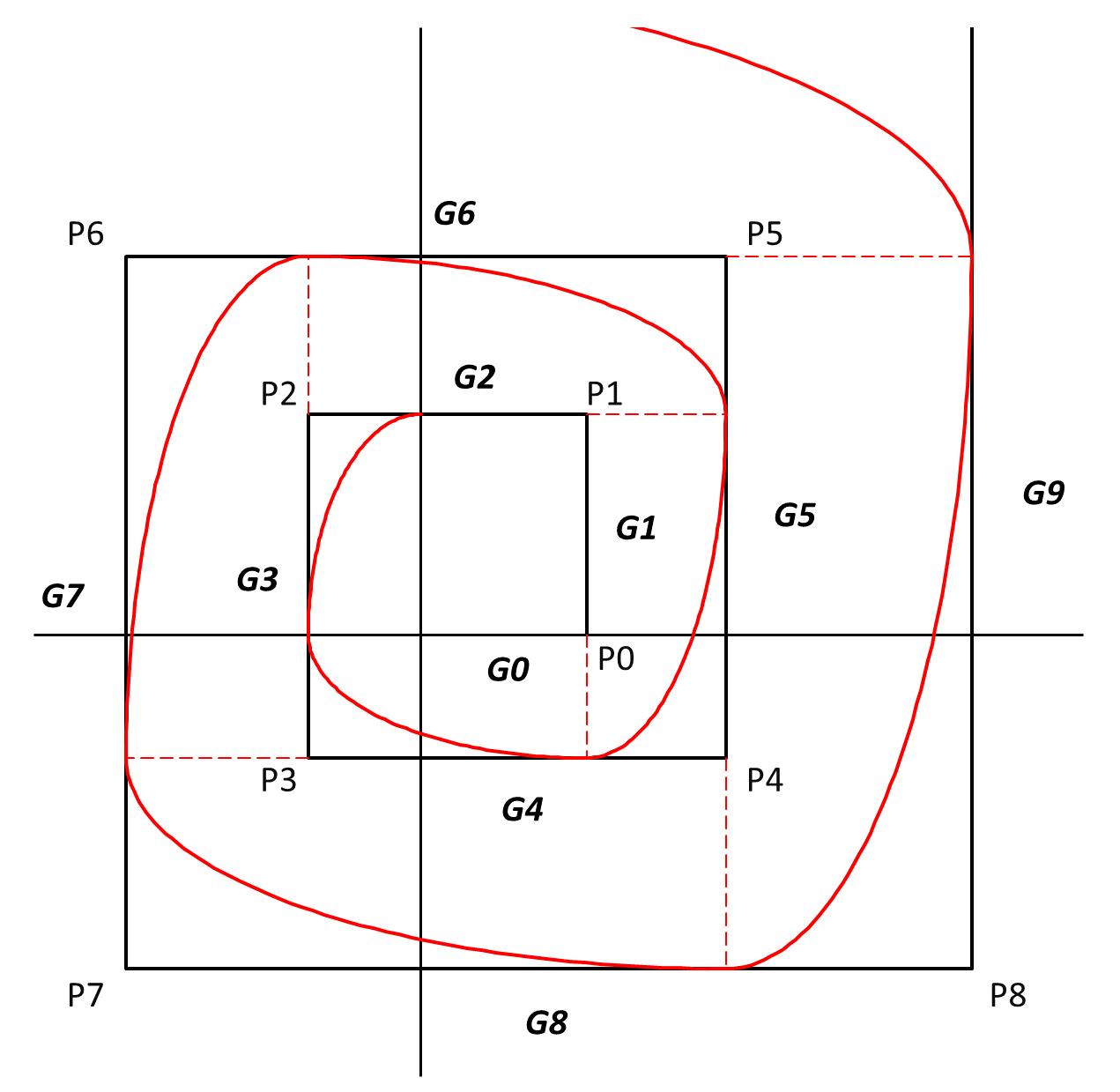}
	\includegraphics[width=0.49\textwidth]{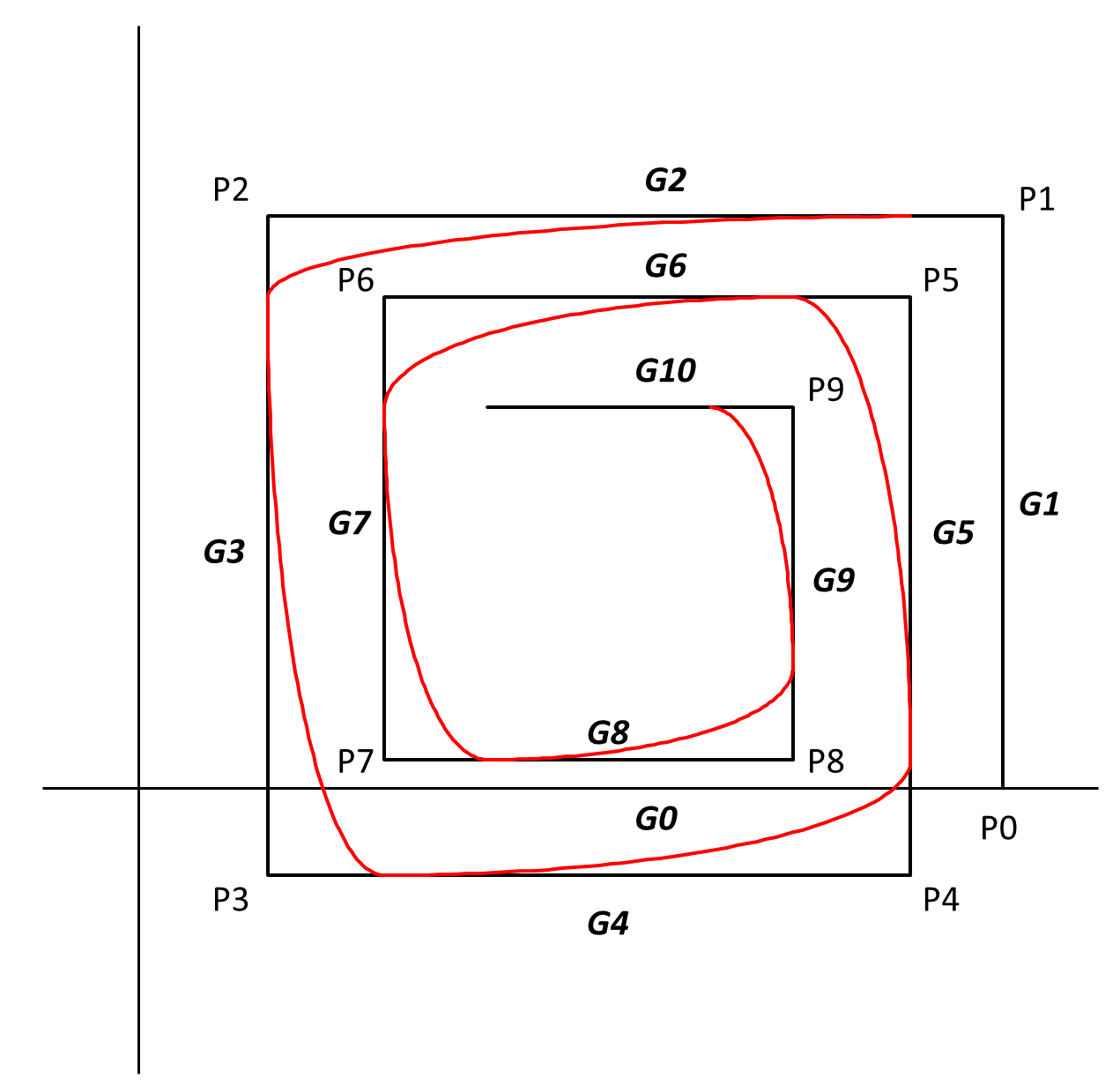}
	\caption{\footnotesize Construction of rectangular and arched generalized Fibonacci spirals (drawn in black and red, respectively), with inserted denotations for partitioned arm lengths $G_n$ (given by a generalized Fibonacci sequence) and corner points $P_n$ (with coordinates determined according to the entries in Tables 1). The rectangular areas used to draw arches by means of quarter-ellipses are indicated through replenished side lines (red dashed lines); these rectangles constitute at the same time a basic tiling pattern. Outwinding spirals are shown in the left panel, inwinding spirals in the right panel. \normalsize}
\end{figure*}
		
\subsection{Rectangular spirals}
\subsubsection{Corner points}
Rectangular genuine Fibonacci spirals are drawn in a two-dimensional Cartesian coordinate system by means of partitioned spiral arms, whose successive linear segments meet at right angles and with lengths equal to successive Fibonacci numbers $F_n=G_n(1,1,0,1;0,1)$ (see Figure 1 for the construction scheme and Figure 2 (upper left panel, black line) for a true-to-scale map). The corresponding corner points $P_n=(X_n;Y_n)$ ($n\in \mathbb{N}_0$) are $P_0=(F_0;0)$, $P_1=(F_0;F_1)$, $P_2=(F_0-F_2;F_1)$, $P_3=(F_0-F_2;F_1-F_3)$, and so on. In general, the coordinates of the corner points are given by sums of alternating even- or odd-indexed Fibonacci numbers $F_n$ according to the entries in Table 1 (top). At every corner point there is a turn to the left by 90$^\circ$, and the distance between any two successive corner points is equal to the related Fibonacci number, i.e., $\overline{P_{n-1}P_n} = F_{n}$.

Similarly, using generalized Fibonacci numbers $G_n$ and adopting relations \eqref{SumAlt} and \eqref{Gamman} for $\Gamma_{n}$, the corner points of a \emph{rectangular generalized Fibonacci spiral} are located at coordinates as listed in Table 1 (bottom). In general, the position of the $n$-th corner point is given by 
\begin{eqnarray}
P_n &=& \big(X_n;\,Y_n \big) \,\,= 
\begin{cases}
\,\,\,\,\big(\,\Gamma_{n}^{}\,\,;\,\,\,\,\Gamma_{n-1}^{} \,\big)
\hspace{0.7cm}n\,\,\,\,\mathtt{even} \\
\,\,\,\,\big(\,\Gamma_{n-1}^{}\,\,;\,\,\,\,\Gamma_{n}^{} \,\big)
\hspace{0.7cm}n\,\,\,\,\mathtt{odd}.
\end{cases} \label{Pn}
\end{eqnarray}
We note that all coordinates $X_n$ go with $\nu = 0$ and all $Y_n$ have $\nu=1$. This reflects the spiral's principle of construction with the $x$- and $y$-coordinates of the corner points being calculated by differences with even-indexed and odd-indexed values of $G_n$, respectively. By construction, the distance relation
\begin{equation}
\overline{P_{n-1}P_n} = G_{n} 
\end{equation}
holds. This can formally be checked by inserting equations \eqref{SumAlt} and \eqref{Pn} into $\overline{P_{n-1}P_n}= [(X_n-X_{n-1})^2+(Y_n-Y_{n-1})^2]^{1/2}$. At every corner point there is a turn to the left by 90$^\circ$ for positive values of $G_n$ (below) or a turn to the right by 90$^\circ$ for negative values of $G_n$.  As an illustration, we shown in Figure 2 (upper panels) a couple of leftwardly outwinding \emph{rectangular generalized Fibonacci spirals} (black lines). They are based on the genuine Fibonacci numbers $F_n=G_n(1,1,0,0;0,1)$ and on the generalized Fibonacci numbers $G_n(0.5,0.8,1,0.9;3,4)$. If instead successive starting points would be directly connected with straight lines, this would produce an \emph{angular spiral} (not shown).

\begin{table}[htp] \label{Table1}
	\footnotesize\centering
	\begin{tabular}{|c|l|l|}
		\hline
		\multicolumn{3}{|c|}{Rectangular Fibonacci spiral$^{}$}\\
		\hline
		$P_n$	&$X_n$ 			&	$Y_n$ 		\\
		\hline
		$P_0$	&$F_0$ 			&	0	  		\\
		$P_1$	&$F_0$ 			&	$F_1$ 	 \\
		$P_2$	&$F_0-F_2$		&	$F_1$      \\
		$P_3$	&$F_0-F_2$		&	$F_1-F_3$ \\
		$P_4$	&$F_0-F_2+F_4$	&	$F_1-F_3$ \\
		$P_5$	&$F_0-F_2+F_4$	&	$F_1-F_3+F_5$\\
		$P_6$	&$F_0-F_2+F_4-F_6$	&	$F_1-F_3+F_5$\\
		$P_7$	&$F_0-F_2+F_4-F_6$	&	$F_1-F_3+F_5-F_7$\\
		\hline
		\multicolumn{3}{|c|}{Rectangular generalized Fibonacci spiral$^{}$}\\
		\hline
		$P_n$	&$X_n$ 			& 	$Y_n$ \\
		\hline
		$P_0$	&$G_0$ 			& 	0   \\
		$P_1$	&$G_0$ 			& 	$G_1$  \\
		$P_2$	&$G_0-G_2$		& 	$G_1$  \\
		$P_3$	&$G_0-G_2$ 		& 	$G_1-G_3$ \\
		$P_4$	&$G_0-G_2+G_4$ 	& 	$G_1-G_3$ \\
		$P_5$	&$G_0-G_2+G_4$	& 	$G_1-G_3+G_5$ \\
		$P_6$	&$G_0-G_2+G_4-G_6$	&	$G_1-G_3+G_5$\\
		$P_7$	&$G_0-G_2+G_4-G_6$	&	$G_1-G_3+G_5-G_7$\\
		$P_n\begin{cases}\tiny{n\,\mathtt{even}}\\
		\tiny{n\,\mathtt{odd}}\end{cases}$ &
		$\begin{cases}\,\,\Gamma_{n}\\  \,\,\Gamma_{n-1}\end{cases}$ &
		$\begin{cases}\,\,\Gamma_{n-1}\\\,\,\Gamma_{n}\end{cases}$ \\
		\hline \multicolumn{3}{c} \normalsize
	\end{tabular}\newline
	\caption{\small Coordinates of the first few \emph{corner points} $P_n=(X_n;\,Y_n)$ for rectangular Fibonacci spirals and for rectangular generalized Fibonacci spirals. For the general case, the involved summation function $\Gamma_{n}$ is given in closed form by equations \eqref{SumAlt} and \eqref{Gamman}. See Figure 1 for an illustration. \normalsize}\normalsize
\end{table} 
\begin{table}[htp] \label{Table2}
	\footnotesize\centering
	\begin{tabular}{|c|c|l|l|}
		\hline
		\multicolumn{4}{|c|}{Arched Fibonacci spiral: quarter-circles$^{}$}\\
		\hline
		arc $\#$& center & $e_x$ 	& $e_y$   	\\
		\hline
		$1$ 	 & $P_{0}$		& $F_1$ 	& $F_1$		\\
		$n\ge 2$ & $P_{n-2}$	& $F_n$ 	& $F_n$		\\
		\hline
		\multicolumn{4}{|c|}{Arched generalized Fibonacci spiral: quarter-ellipses$^{}$}\\
		\hline
		arc $\# $& center & $e_x$ 	& $e_y$   	\\
		\hline
		1 & $P_0$   & $G_2-G_0$ & $G_1$		\\
		\multicolumn{4}{|l|}{outwinding:$^{}$}\\
		$\begin{cases}n\,\mathtt{even}\\
		n\,\mathtt{odd}\end{cases}$ & $P_{n-2}$ &
		$\begin{cases} G_n\\G_{n+1}-G_{n-1} \end{cases}$ &		$\begin{cases} G_{n+1}-G_{n-1}\\G_n \end{cases}$ \\
		\multicolumn{4}{|l|}{inwinding:$^{}$}\\
		$\begin{cases}n\,\mathtt{even}\\
		n\,\mathtt{odd}\end{cases}$ & $P_{n+4}$ &
		$\begin{cases} G_{n+4}-G_{n+2} \\G_{n+3} \end{cases}$ &		$\begin{cases} G_{n+3} \\G_{n+4}-G_{n+2}  \end{cases}$ \\
		\hline  \multicolumn{3}{c} \normalsize
	\end{tabular}\newline\newline
	\caption{\small Quarter ellipses for arched generalized Fibonacci spirals: the \emph{center points} and lengths of the semi-axes $e_x$ and $e_y$ in $x$- and $y$-direction, respectively, for the $n$th arc. The center points of the arcs coincide with particular corner points $P_i$ for rectangular spirals (given in Table 1). For an illustration, see Figure 1. Arched Fibonacci spirals represent a special case with $F_n=G_n(1,1,0,0;0,1)$. Because Fibonacci numbers obey the identity $F_{n+1}-F_{n-1}=F_n$, their quarter ellipses reduce to quarter circles with fixed radii $F_n$. Figure 2 (upper left panel) provides an illustration.\normalsize}\normalsize
\end{table}		
\begin{figure*}	\label{Fig_2}
	\centering 
	\includegraphics[width=0.45\textwidth]{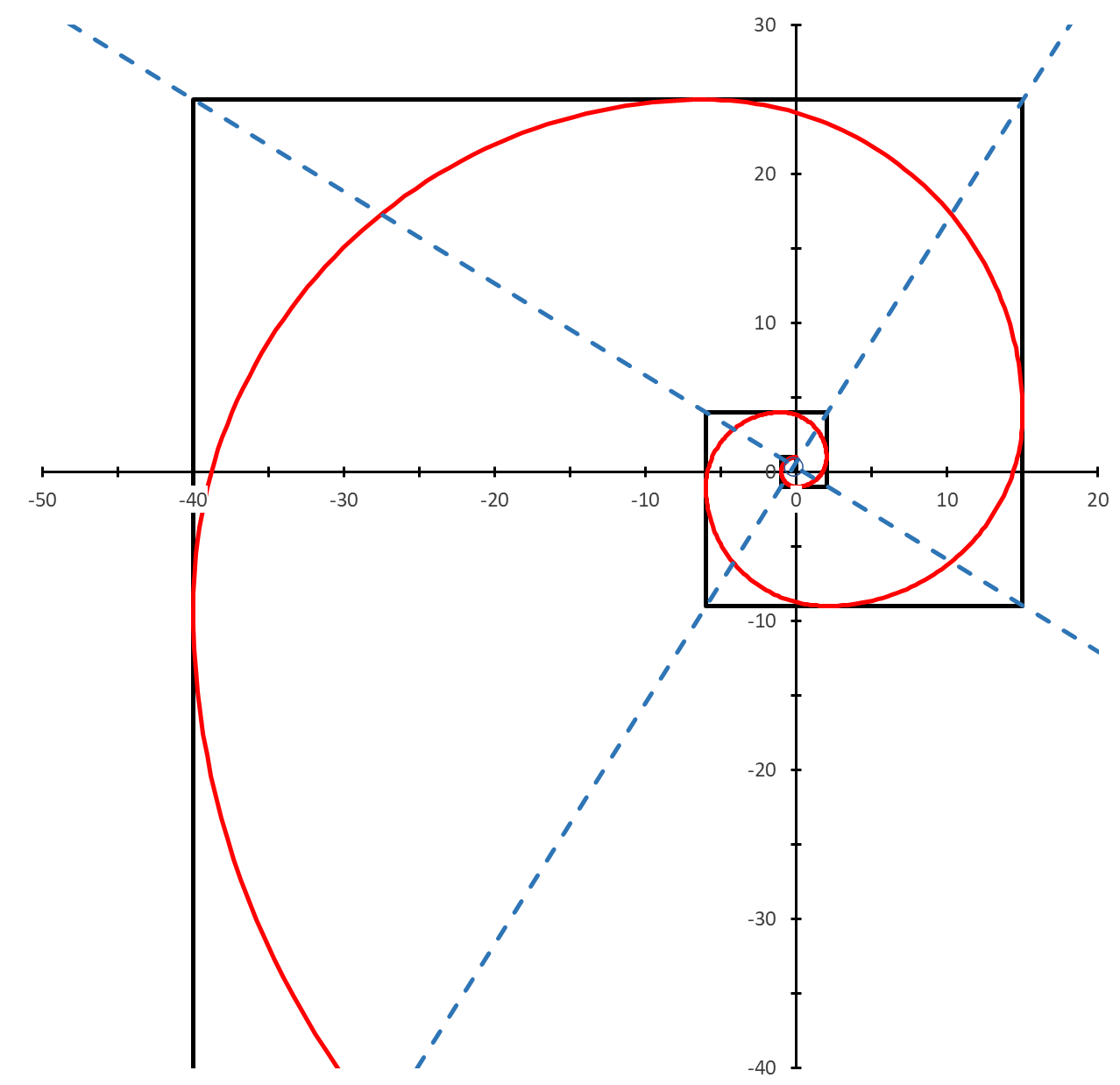}
	\includegraphics[width=0.45\textwidth]{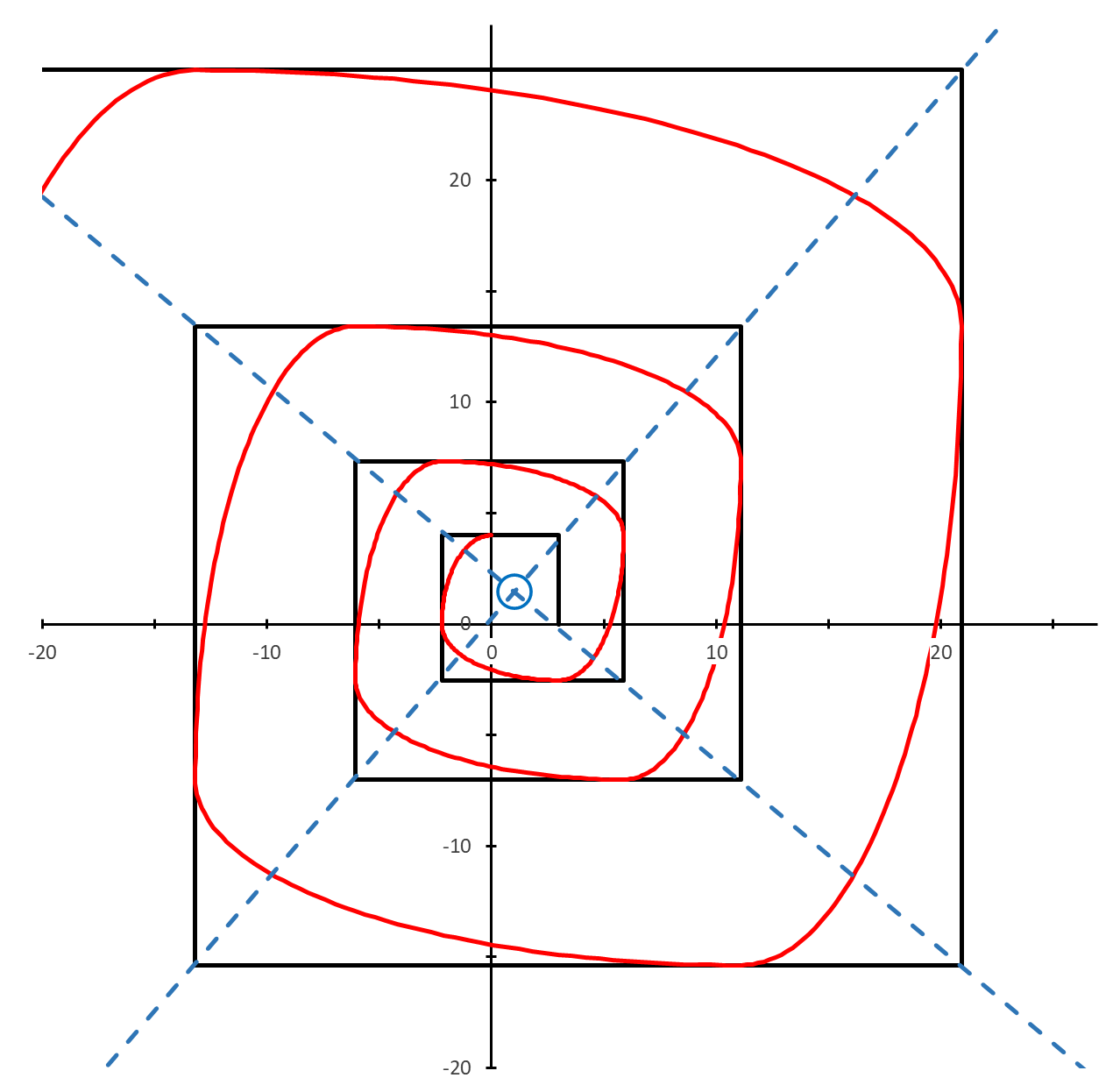}
	\includegraphics[width=0.45\textwidth]{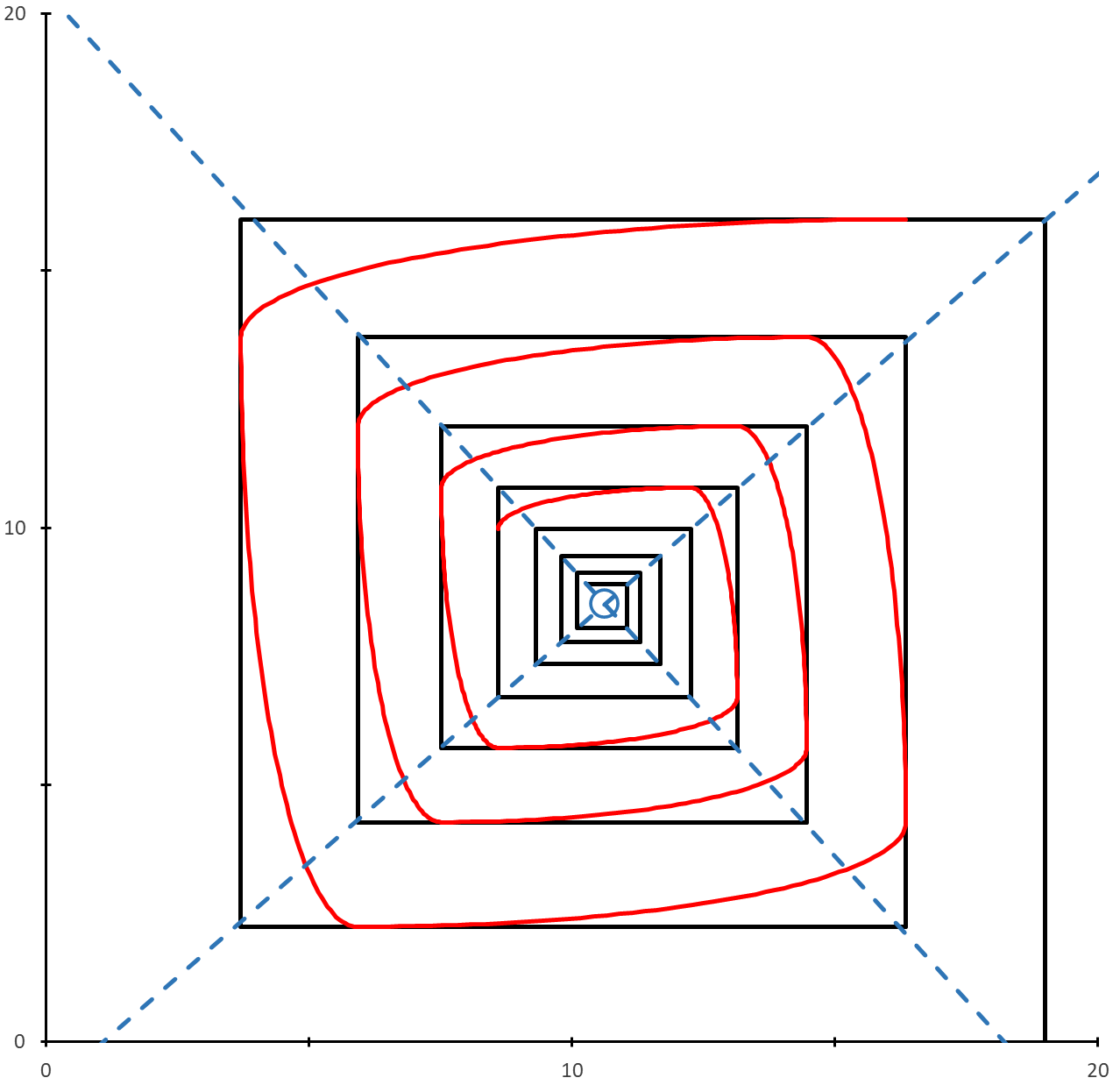}
	\includegraphics[width=0.45\textwidth]{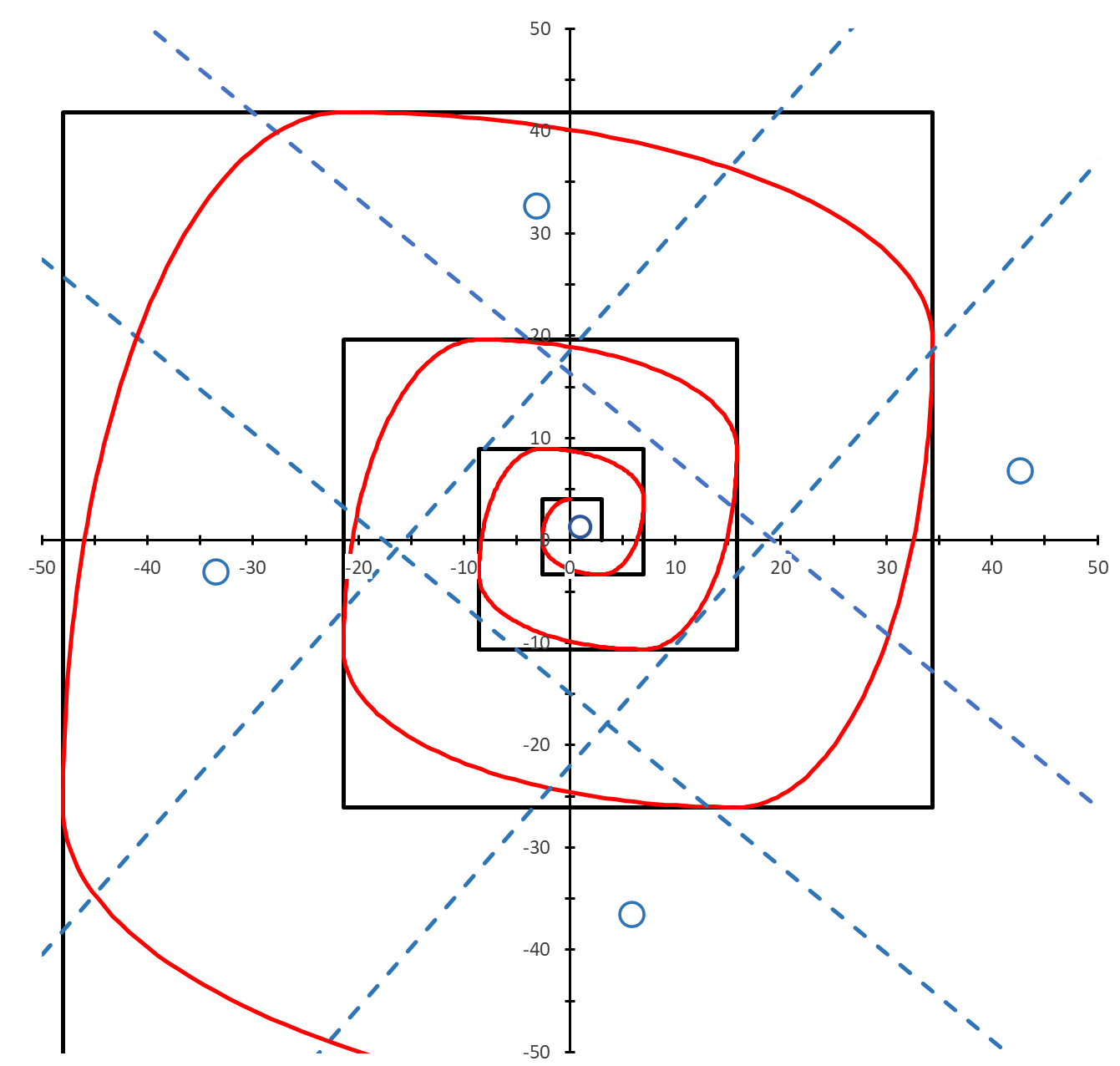}
	\caption{\footnotesize Examples of rectangular and arched generalized Fibonacci spirals. \newline
		\textit{Upper left}: Genuine Fibonacci spirals tracing the Fibonacci numbers $F_n=G_n(1,1,0,0;\,0,1)$ ($n\in \mathbb{N}_0$) by means of either perpendicular arms with lengths 0, 1, 1, 2, 3, 5, 8, 13, $\dots\,$ (black lines) or successive quarter-circles with corresponding radii (red). The oblique asymptotes for the directional corner points (blue dashed lines) are orthogonal, with slopes $\Phi\approx1.618$ (i.e., the golden ratio) and $-1/\Phi\approx -0.618$, and intersect at the point $P^*=(-1/5;\,2/5)$.
		\textit{Upper right}: Outwinding generalized Fibonacci spirals with $G_n(0.5,0.8,1,0.9;\,3,4)$, both rectangular and elliptically arched. The point $P^*\approx(1.033; \,1.502)$ coincides with the point of intersection for the asymptotes of directional corner points (shown dashed for the lines along $P_{48}P_{52}$ and $P_{49}P_{53}$). \textit{Lower left}: Inwinding generalized Fibonacci spirals with $G_n(0.5,0.35,1,0.8;\,19,16)$. The point of convergence, located at $P^*\approx(10.63;$\,$8.52)$, meets the point of intersection of the asymptotes.
		\textit{Lower right}: Outwinding generalized Fibonacci spirals for $G_n(0.5,0.8,1,1.1;\,3,4)$ (differing only in the parameter value for $d$ with respect to the upper right figure): because now $\alpha>1$ and $d>1$, the points of intersection for the asymptotes (here represented by the dashed lines $P_{48}P_{52}$, $P_{49}P_{53}$, $P_{50}P_{54}$, and $P_{51}P_{55}$) drift away from $P^\star$ (central circle). The other four circles in the vicinity give the positions of approximately calculated points of intersection (see text for details). \normalsize}
\end{figure*}

\subsubsection{Orthogonal asymptotes}

Assuming throughout this subsection positive initial values $G_0$ and $G_1$, the corner points $P_0$, $P_4$, $P_8$, $\dots$, $P_{4k}$, $\dots$ ($k\in\N_0$) constitute the infinite set of all lower-right corner points. In a similar manner, the sequences $\{P_{1+4k}\}$, $\{P_{2+4k}\}$, and $\{P_{3+4k}\}$ form the sets of upper-right, upper-left, and lower-left corner points, respectively (see Figure 1). In general, one has sets of \textit{directional corner points} $\{P_{j+4k}\}$, where $k\in\N_0$ is the running number and the constant value $j\in\{0,1,2,3 \}$ fixes the direction. Stated differently, any corner point $P_n$ corresponds to a directional corner point with directional index $j=n\,\mathtt{mod}\,4$. The slope of a straight line between two neighbouring directional corner points $P_{n}$ and $P_{n+4}$ (implying an equal index value $j$ and with  coordinates as given by equation \ref{Pn}) is determined by
\begin{equation}
\frac{Y_{n+4}-Y_{n}}{X_{n+4}-X_{n}}=
\begin{cases}
\,\,\frac{\Gamma_{n+3}-\Gamma_{n-1}}{\Gamma_{n+4}-\Gamma_{n}}
\hspace{1.0cm}n\,\,\,\mathtt{even}\\
\,\,\frac{\Gamma_{n+4}-\Gamma_{n}}{\Gamma_{n+3}-\Gamma_{n-1}}
\hspace{1.0cm}n\,\,\,\mathtt{odd}.
\end{cases}
\end{equation}
In the very large $n$-limit this becomes  
\begin{equation}
n\gg1:\hspace{0.5cm}\frac{Y_{n+4}-Y_{n}}{X_{n+4}-X_{n}}\,\,\,\rightarrow\,\,\,
\begin{cases}
\,\,\,-1/\gamma\hspace{1.0cm}
n\,\,\,\mathtt{even}\\
\,\,\,\hspace{0.5cm}\gamma\hspace{1.2cm} n\,\,\,\mathtt{odd}.
\end{cases}\hspace{1.5cm}
\end{equation}
\begin{proof}
Assuming $\Gamma_{n+4}\gg\Gamma_{n}$, $\Gamma_{n+3}\gg\Gamma_{n-1}$ (for $n\gg1$) and adopting relation \eqref{Gammanlarge} one has 
\begin{equation*}\frac{Y_{n+4}-Y_{n}}{X_{n+4}-X_{n}}=
\begin{cases}
	\,\,\frac{\Gamma_{n+3}-\Gamma_{n-1}}{\Gamma_{n+4}-\Gamma_{n}}
	\approx\frac{\Gamma_{n+3}}{\Gamma_{n+4}} \approx
	\frac{(-1)^{(n+3-1)/2}\gamma^{n+5}}{(-1)^{(n+4-0)/2}\gamma^{n+6}} = -\frac{1}{\gamma}
	\hspace{1.0cm}n\,\,\,\mathtt{even} 
	\\
	\,\,\frac{\Gamma_{n+4}-\Gamma_{n}}{\Gamma_{n+3}-\Gamma_{n-1}}
	\approx\frac{\Gamma_{n+4}}{\Gamma_{n+3}} \approx
	\frac{(-1)^{(n+4-1)/2}\gamma^{n+6}}{(-1)^{(n+3-0)/2}\gamma^{n+5}} = \gamma
	\hspace{1.35cm}n\,\,\,\mathtt{odd}.
\end{cases}
\end{equation*}
Assuming instead $\Gamma_{n+4}\ll\Gamma_{n}$, $\Gamma_{n+3}\ll\Gamma_{n-1}$ (for $n\gg1$) one similarly gets 
\begin{equation*}\frac{Y_{n+4}-Y_{n}}{X_{n+4}-X_{n}}=
	\begin{cases}
		\,\,\frac{\Gamma_{n+3}-\Gamma_{n-1}}{\Gamma_{n+4}-\Gamma_{n}}
		\approx\frac{\Gamma_{n-1}}{\Gamma_{n}} \approx
		\frac{(-1)^{(n-1-1)/2}\gamma^{n+1}}{(-1)^{(n-0)/2}\gamma^{n+2}} = -\frac{1}{\gamma}
		\hspace{1.0cm}n\,\,\,\mathtt{even} 
		\\
		\,\,\frac{\Gamma_{n+4}-\Gamma_{n}}{\Gamma_{n+3}-\Gamma_{n-1}}
		\approx\frac{\Gamma_{n}}{\Gamma_{n-1}} \approx
		\frac{(-1)^{(n-1)/2}\gamma^{n+2}}{(-1)^{(n-1-0)/2}\gamma^{n+1}} = \gamma
		\hspace{1.35cm}n\,\,\,\mathtt{odd}.
	\end{cases}
\end{equation*}
\end{proof}
This asymptotic behaviour immediately implies the following proposition.
\begin{proposition}\label{Prop1} 
For rectangular generalized Fibonacci spirals, the oblique asymptotes for the directional corner points with either even- or odd-numbered indices lie mutually orthogonal.
\end{proposition} 
\begin{proof}
In general, the (complanary) graphs of two linear functions with slopes $s_1$ and $s_2$ are perpendicular with respect to each other, if the condition $s_1s_2=-1$ holds. This is satisfied with the two asymptotic slopes $-1/\gamma$ and $\gamma$.
\end{proof}  
As special case, the rectangular genuine Fibonacci spiral goes with asymptotic slopes through directional corner points that are equal to the Golden ratio and its negative inverse, i.e.,  $\gamma=\lim_{n\rightarrow\infty}F_{n+4}/F_{n+3}= \lim_{n\rightarrow\infty} F_{n+1}/F_{n}=\Phi=1.618\dots$ ($n$ odd) and $-1/\gamma=-\Phi^{-1}=-0.618\dots$ ($n$ even). For its exhibition, see Figure 2 (upper left panel). The two asymptotes drawn lead through the directional corner points with $n$=48 and 52 and with $n$=49 and 53 (representing some large-$n$-limit) and intersect at the point $P^\star=(-1/5;\,2/5)$. Herein, the coordinates are calculated by means of the formula given below in Proposition \ref{Prop2}.

As another special case, for rectangular Horadam spirals ($c=0$) the corresponding asymptotic slopes are $\gamma=\alpha$ and  $-1/\gamma=-1/\alpha$ (equivalent to an eigenvalue of the Horadam  matrix \eqref{Gmatrix1} refered to in the Appendix). 

Proposition \ref{Prop1} holds for both types of windings, i.e., for inwinding and outwinding spirals, however, with some particularities, as emphasized by the following proposition. 
\begin{proposition}\label{Prop2} 
For outwinding rectangular spirals with either $\gamma=\alpha>1>d>0$ or $\gamma=d>1>\alpha>0$ and for inwinding spirals ($\gamma<1$) the two perpendicular asymptotes for directional corner points cross at the point of intersection $P^\star=(X^\star;\,Y^\star)$ given by
\begin{equation}\label{Pstar}
P^*=\frac{1}{a^2+(b+1)^2} \Big( 
G_0 +b^2G_{-2}-\,c\,\frac{d^2+ad-b}{d^2+1} \,\,;\,\, 
G_1 +b^2G_{-1}-\,c\,d\,\frac{d^2+ad-b}{d^2+1} \Big);
\end{equation}
for inwinding spirals this coincides with the point of convergence $\lim_{n\rightarrow\infty} P_n$.
\end{proposition}
\begin{proof}
(Sketch) The claimed coordinates for the point of convergence of \emph{inwinding} spirangles directly follow by inserting relation \eqref{Gammanlarge} (case $\gamma<1$) into equation \eqref{Pn}. Due to Proposition \ref{Prop1} this must be equal to the point of intersection. For \emph{outwinding} spirangles with $\gamma>1$ subsequent corner points $P_n=(\Gamma_n;\,\Gamma_{n-1})$ ($n$ even, and for the moment with directional index $j$=0) and $P_{n+1}=(\Gamma_{n+1};\,\Gamma_n$) ($n+1$ odd, $j$=1) lie in the very large-$n$-limit on orthogonal lines described by linear functions $y=-\gamma^{-1}\,x+(\Gamma_{n-1}+\gamma^{-1}\,\Gamma_n)$ and $y=\gamma\,x+(\Gamma_{n+1}-\gamma\,\Gamma_n)$, respectively. Hereby the slopes are chosen according to Proposition \ref{Prop1}. Equating the functions yields $x\equiv X_{n, j=0}^\star$ and $y\equiv Y_{n, j=0}^\star$ as given by $X_{n, j=0}^\star=\Gamma_n-\gamma(\gamma^2+1)^{-1} \,(\Gamma_{n+1}-\Gamma_{n-1}) =\Gamma_n-\gamma(\gamma^2+1)^{-1} \,G_{n+1}$ and $Y_{n,j=0}^\star=-\gamma^{-1}\,X_n^\star+(\Gamma_{n-1}+\gamma^{-1}\,\Gamma_n) =\Gamma_{n-1}+(\gamma^2+1)^{-1}\,G_{n+1}$. Adopting either $\gamma=\alpha>1>d$ or $\gamma=d>1>\alpha$  and taking $G_n\approx A\alpha^n$ or $G_n\approx pd^n$, respectively (due to relation \ref{Gammanlarge}), inserting expression \ref{Gamman} and using relations (2.6)-(2.8) straightforwardly provides the reclaimed point $P^\star$. Similarly, if one starts the inquiry with directional indices $j$=1,2,3 (instead of $j$=0 as above), one gets approximate points of intersection 
($X_{n, j=1}^\star,Y_{n,j=1}^\star)=
(\Gamma_n-(\gamma^2+1)^{-1} \,G_{n+2};\,
\Gamma_{n+1}-\gamma(\gamma^2+1)^{-1}\,G_{n+2})$, 
($X_{n, j=2}^\star,Y_{n,j=2}^\star)=
(\Gamma_{n+2}+\gamma(\gamma^2+1)^{-1} \,G_{n+3};\,
\Gamma_{n+1}-(\gamma^2+1)^{-1}\,G_{n+3})$, and 
($X_{n,j=3}^\star,Y_{n,j=3}^\star)=
(\Gamma_{n+2}+(\gamma^2+1)^{-1} \,G_{n+4};\,
\Gamma_{n+3}+\gamma(\gamma^2+1)^{-1}\,G_{n+4})$, respectively. Adopting the procedure as above for $j=0$, the conclusion arrived before concerning $P^\star$ remains the same in all cases.
\end{proof}

As with the rectangular Fibonacci spiral shown in Figure 2 (upper left panel), the asymptotes for the rectangular generalized Fibonacci spirals shown in the upper right panel (outwinding) and in the lower left panel (inwinding) are approximated by the blue-dashed lines going through points $P_{48}$ (and $P_{52}$) and $P_{49}$ (and $P_{53}$). In all examples the point of intersection $P^\star$ is calculatd with equation \eqref{Pstar} and marked by a blue circle. 

For differently outwinding rectangular spirals, i.e., those with both $\alpha>1$ and $d>1$, the slopes of lines through two neighbouring directional corner points still obey Proposition \ref{Prop1}. But in this case the mutual points of intersection $P_{n,j}^*=(X_{n,j}^\star;\, Y_{n,j}^\star)$ ---as given approximately in the proof above--- do not converge, with increasing value of $n$ they instead diverge away from $P^\star$. A zoom-in illustration is given in the lower right panel of Figure 2. The points of intersection (crossing dashed lines) are loacated only in the vicinity of $P^\star$ (single central circle). Their drift distances are, however, much smaller than the respective coordinates of the directional corner points: if the axes in this lower right panel would be scaled-out to the value of, say, $P_{48}=(17478;\,	-14798)$, all four asymptotes would pass $P^\star$ invisible close. The reader is invited to check this herself with the help of a graphical tool.

These mutually orthogonal lines are closely related to what may be called \textit{Holden lines}, because a similar feature of orthogonality was already recognized by Holden (1975) \cite{holden75} for the successive centers of quadratic tiles related to an outwinding genuine Fibonacci spiral. Hoggatt \& Alladi (1976) \cite{hoggatt76} modified the results of Holden by means of an inwinding rectangular Horadam spiral defined by the sequence $G_n(1,1,0,0;1,G_1)$ and found the intersection of the asymptotes of the directional corner points to correspond to the point of convergence for their spiral. Except for a rotation of the coordinate system, their observation seems to encounter a special case of our generalized approach.

Last but not least, we note that the length $L_n$ of a rectangular generalized Fibonacci spiral that starts at the origin of the coordinate system can be calculated by either one of the sum formulae \eqref{sumGk1} to \eqref{sumGk3}. In case of an inwinding spiral with infinitely many segments, the finite total length becomes
\begin{equation}
L_\infty = \frac{1}{a+b-1} \Bigg( (a-1)G_0-G_1
-\frac{cd^2}{1-d}\Bigg).
\end{equation}

\subsection{Arched spirals}

As can be inferred from Figure 1 and Table 2, the starting point of the $n$-th quarter-circle of an arched Fibonacci spiral has radius $F_n$ and the center is located at corner point $P_{n-2}$ (with the coordinates according to equation \ref{Pn}). Modestly more complicated, the $n$-th quarter-ellipse of an arched generalized Fibonacci spiral has unequally long semi-axes and depends on the even/odd distinction for $n$. As with arched genuine Fibonacci spirals, the center is identical with corner point $P_{n-2}$ for outwinding spirals. However, for inwinding spirals the center is chosen here to be identical with corner point $P_{n+4}$. (Alternatively, another inviting choice would have been $P_{n+6}$.) In order to draw the arc of the $n$-th quarter-ellipse or -circle, some $N+1$ \emph{spiral points} $P_{n,i}$, $i=0,1,2,\dots,N$,  with Cartesian coordinates $(x_{n,i};\,y_{n,i})$ on the arc are calculated and interpolated. We use parameter representation to have spiral points
\begin{equation}\label{Pniout}
P_{n,i}^{\mathtt{out}}=(x_{n,i};\,y_{n,i}) = 
\begin{cases}\,\,\,
\big(\Gamma_{n-2}^{}+G_n\cos \phi_{n,i}\,;
\,\,\,\,\,\Gamma_{n-3}^{}+(G_{n+1}-G_{n-1})\sin \phi_{n,i}\big) \hspace{0.7cm}n\,\mathtt{even}\\ \,\,\,
\big(\Gamma_{n-3}^{}+(G_{n+1}-G_{n-1})\cos \phi_{n,i}\,;
\,\,\,\,\,\Gamma_{n-2}^{}+G_n\sin \phi_{n,i}\big)  \hspace{0.7cm}n\,\mathtt{odd},
\end{cases}
\end{equation}
for \textit{outwinding} spirals and 
\begin{equation}
P_{n,i}^{\mathtt{in}}=(x_{n,i};\,y_{n,i}) = \begin{cases}\,\,\,
\big(\Gamma_{n+4}^{}+\lvert G_{n+4}-G_{n+2}\rvert\cos \phi_{n,i}\,;
\,\,\,\,\,\Gamma_{n+3}^{}+G_{n+3}\sin \phi_{n,i}\big) \hspace{0.5cm}n\,\mathtt{even}\\ \,\,\,
\big(\Gamma_{n+3}^{}+G_{n+3}\cos \phi_{n,i}\,;
\,\,\,\,\,\Gamma_{n+4}^{}+\lvert G_{n+4}-G_{n+2}\rvert\sin \phi_{n,i}\big)  \hspace{0.5cm}n\,\mathtt{odd},\label{Pni}
\end{cases}
\end{equation}
for \textit{inwinding} spirals, in both cases with attributed polar angles 
\begin{equation}
\phi_{n,i}=\bigg(n+\frac{i}{N}\bigg)\frac{\pi}{2}\,\,\,\in\,\bigg[ n\frac{\pi}{2}, \,(n+1)\frac{\pi}{2} \bigg],\hspace{1cm}i=0,1,2,\dots,N.\label{phiniin}
\end{equation}
For our figures, we choose $N=60$. By definition, the \emph{starting points} of the first and of the second arc are $P_{1,0}=(0;\,G_1)$ and $P_{2,0}=(G_0-G_2;\,0)$, respectively. (This may be modified in another study.) For $n\ge3$  and with $\phi_{n,0}=n\,\pi/2$ ($i=0$ in equation \ref{Pni}), the starting point of the $n$-th arc is 
\begin{equation}\label{Pinterpol0}
P_{n,0}=(x_{n,0};\,y_{n,0}) = \begin{cases}\,\,\,
\big(\Gamma_{n-2}^{}+G_n (-1)^{n/2}\,;
\,\,\,\,\,\Gamma_{n-3}^{}\big)\,\,
\hspace{0.5cm}=\big(\Gamma_{n}^{};\,\Gamma_{n-3}^{}\big) \hspace{0.7cm}n\,\mathtt{even}\\ 
\,\,\,\big(\Gamma_{n-3}^{}\,;\,\,\,\,\,\, \Gamma_{n-2}^{}+G_n(-1)^{(n-1)/2} \big)
=\big(\Gamma_{n-3}^{};\,\Gamma_{n}^{}\big)  \hspace{0.68cm}n\,\mathtt{odd},
\end{cases}
\end{equation}
where the rear equalities directly follow from equation \eqref{diffrel}.

The assemblage of arcs of subsequent quarter-ellipses composes an \emph{arched generalized Fibonacci spiral} in the plane. Some example spirals, including the popular planar Fibonacci spiral, are shown in Figure 2 (red curves). 

Some further observations concerning large values of $n$ are as follows:
\begin{itemize}
	\item As already stated above, for large $n$ the value of $G_n=A\,\alpha^n+B\,\beta^n+p\,d^{n}$ ($\alpha$, $\beta$, $d$ being distinct) is dominated by  $\gamma\,$=$\,\max\{\lvert\alpha\rvert, \lvert\beta\rvert, \lvert d\rvert\}$, i.e., $G_n\propto \gamma^n$. 
	\item The \emph{ellipticity} of the $n$th arc (with $n$ even) is defined by  $\epsilon_{out}=1-(G_{n+1}-G_{n-1})/G_n$ for outwinding spials and $\epsilon_{in}=1-G_{n+3}/\rvert G_{n+4}-G_{n+2}\lvert$ for inwinding spirals. For large $n$, the ellipticities approach constant values $1-\gamma$ and $1-1/\gamma$, respectively. A similar statement can be made for the case with $n$ being odd.	
	\item For large $n$, the spiral points are approximated with increasing accuracy by
\begin{equation}\label{Pnilargen}
n\,\,\mathtt{large}:\,\,\,P_{n,i}\approx \begin{cases}\,\,\,
\big(\Gamma_{n-2}^{}+\gamma^n\cos \phi_{n,i}\,;
\,\,\,\,\,\Gamma_{n-3}^{}+\gamma^{n+1}\sin \phi_{n,i}\big) \hspace{0.7cm}n\,\mathtt{even}\\ \,\,\,
\big(\Gamma_{n-3}^{}+\gamma^{n+1}\cos \phi_{n,i}\,;
\,\,\,\,\,\Gamma_{n-2}^{}+\gamma^n\sin \phi_{n,i}\big)  \hspace{0.7cm}n\,\mathtt{odd}.
\end{cases}
\end{equation}
For Fibonacci numbers $F_n=G_n(1,1,0,0;0,1)$ or Lucas numbers $L_n=G_n(1,1,0,0;2,1)$, both with $\gamma=\Phi\approx 1.618$ being equal to the golden ratio, this is known to approximately become a particular \emph{geometric spiral} (a.k.a. logarithmic or Bernoulli spiral). Such a spiral is investigated in, e.g., \cite{engstrom87}.
\end{itemize}
\begin{figure*}	\label{Fig_3}
	\centering 
	\includegraphics[width=0.45\textwidth]{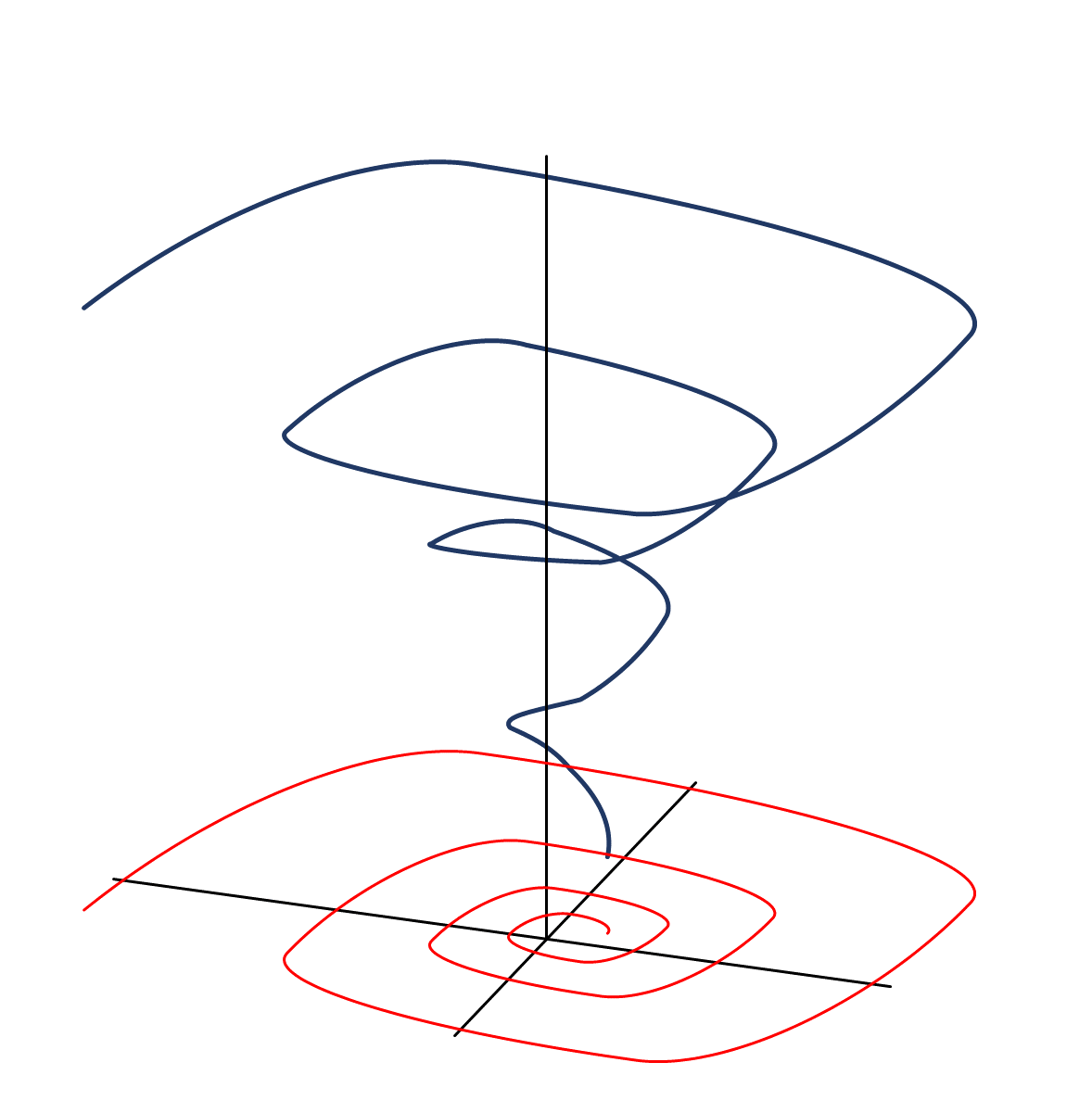}
	\includegraphics[width=0.45\textwidth]{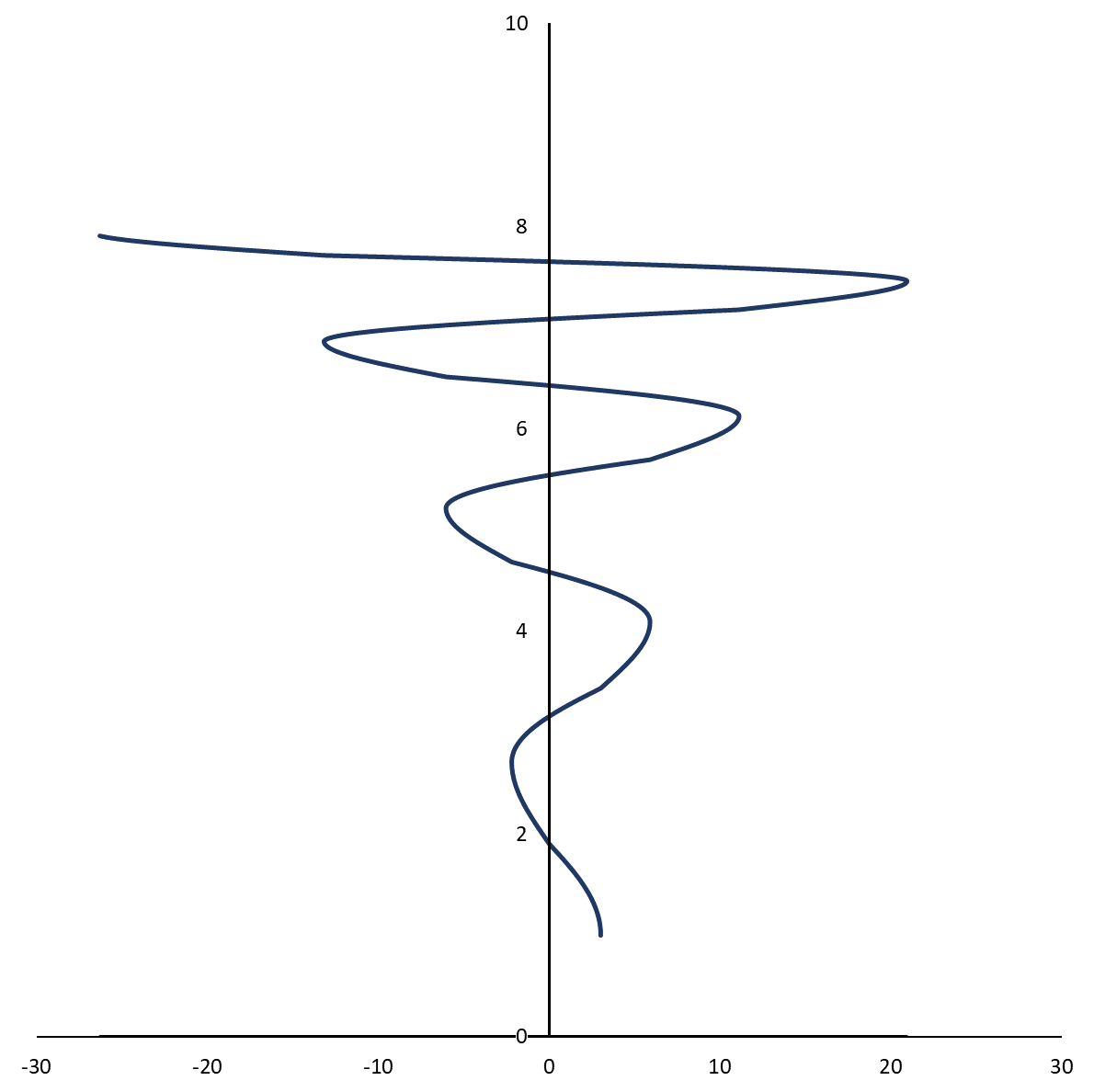}
	\caption{\footnotesize Spacial generalized Fibonacci spiral.
		\textit{Left}: Perspective view of the outwinding spiral shown in Figure 2 (upper right panel, reproduced here in red), with the $z$-coordinate given by the cumulated inputs according to equation \eqref{z3}. \textit{Right}: Same spiral as shown in the panel to the left, but seen projected onto the $x$-$z$-plane in a frontal view. For $n\rightarrow \infty$ the height of the spiral will converge to the upper limit $z_\infty=c/(1-d)=10$ (with $c=1,\,d=0.9$). \normalsize}
\end{figure*}

\subsection{Spacial representation}

We finally add a third dimension to \emph{arched} spirals to get interpolated points with spacial coordinates $P_{n,i}=(x_{n,i};\,y_{n,i};\,z_{n,i})$, with $x_{n,i}$, $y_{n,i}$ still calculated according to equations \eqref{Pniout}-\eqref{Pinterpol0}. For the third coordinate $z_{n,i}$, $i=0,1,2,\dots,N-1$, one may simply choose some linear increase, i.e., 
\begin{equation}\label{z1}
z_{n,i}= n+\frac{i}{N},
\end{equation} 
where $N$ is the number of points used to interpolate each quarter-ellipse. Or one may visualize the exponential input that appears in recursion relation \eqref{RR} by exhibiting either the local contribution
\begin{equation}\label{z2}
z_{n,i} = c\,d^{\,n+i/N}
\end{equation}
or the cumulated inputs 
\begin{equation}\label{z3}
z_{n,i} = c\,\frac{d^{\,n+1+i/N}-1}{d-1}\,, 
\end{equation}
where for $i=0$ one has in the latter case $z_{n,i}=c\,S_n(d)$, $S_n(d)$ being the partial sum formula for geometric progression (see equation \ref{Sumnd}). Alternatively, replacing $c$ by $p$ will provide an altered information. A typical \emph{spacial spiral} is exhibited in Figure 4, seen under distinct perspectives. It is but the outwinding arched spiral already shown in Figure 2 (upper right panel), with the $z$-coordinate calculated according to the rule given by equation \eqref{z2}. Because $G_n$ is growing with increasing index number $n$, the spiral arms get exponentially broader, while the height of this spiral approaches an upper limit (see Figure caption). 

\begin{figure*}	\label{Fig_Gt}
	\centering 
	\includegraphics[width=0.45\textwidth]{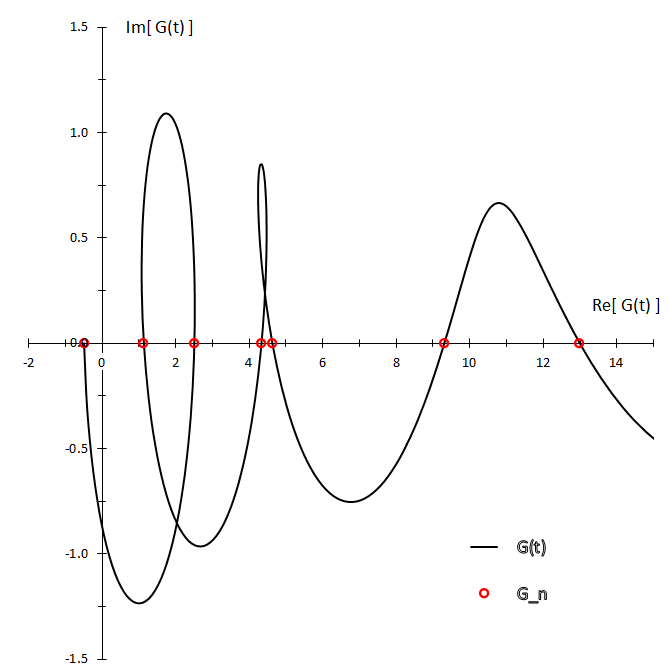}
	\includegraphics[width=0.45\textwidth]{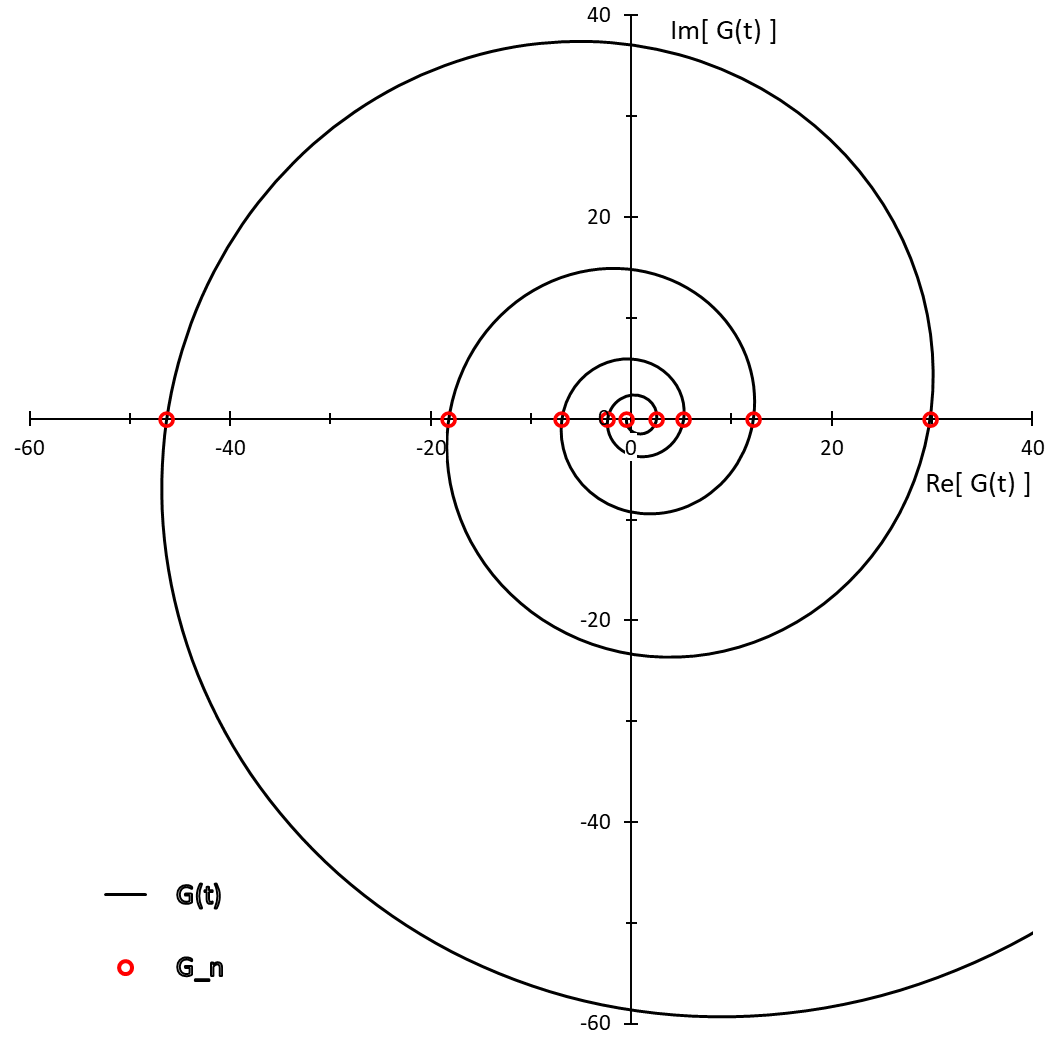}
	\includegraphics[width=0.45\textwidth]{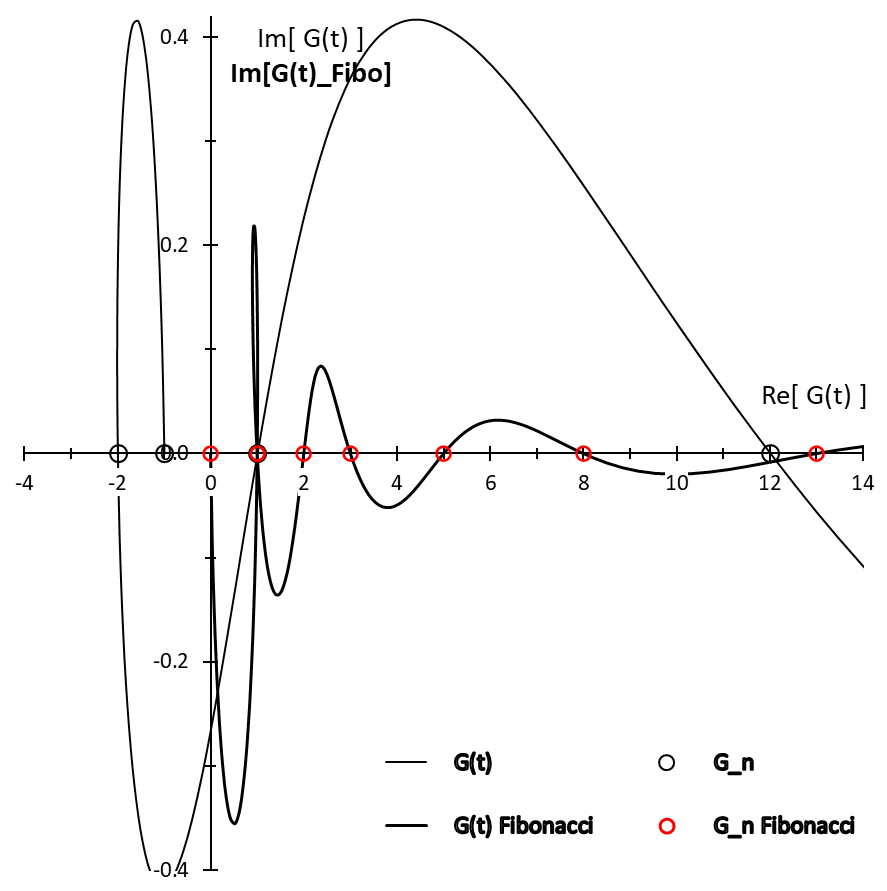}
	\includegraphics[width=0.45\textwidth]{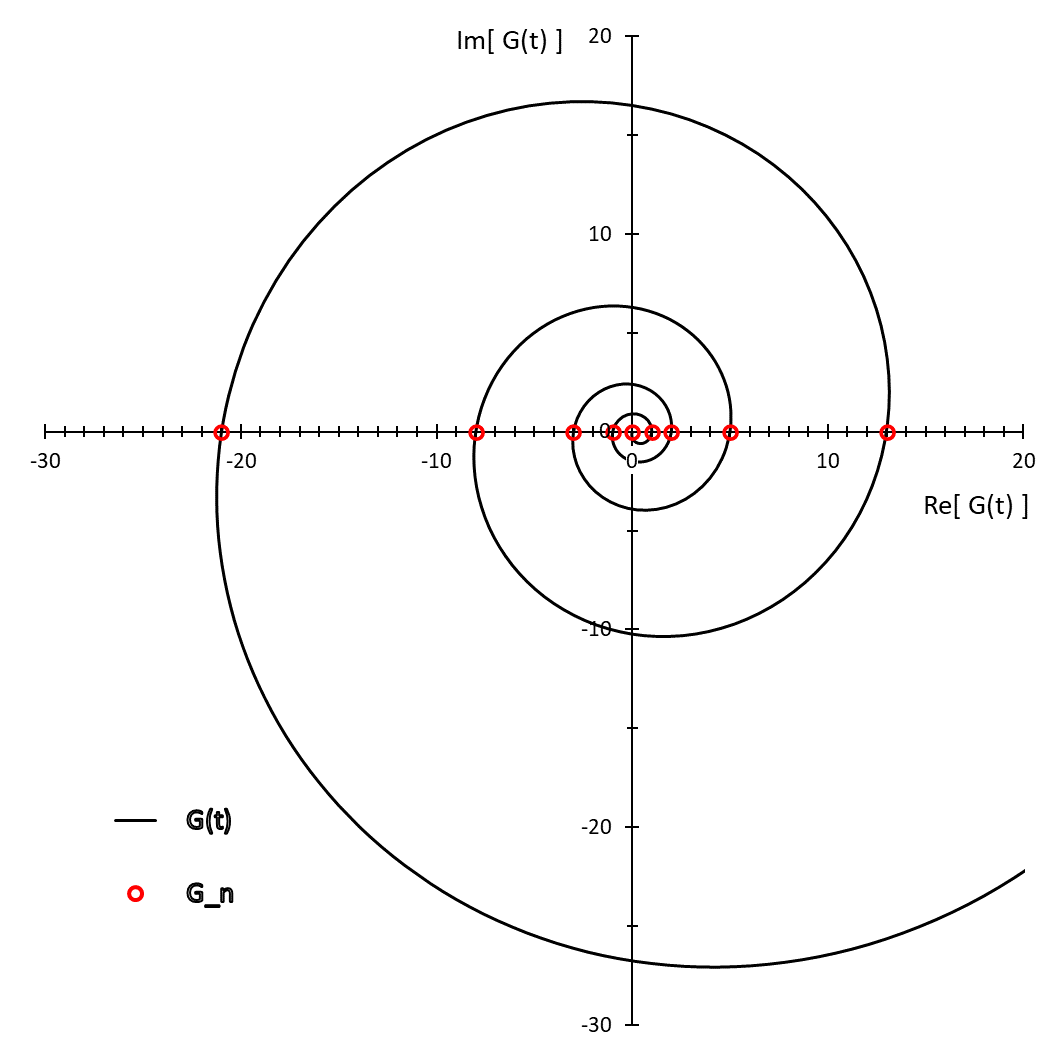}
	\caption{\footnotesize Illustrating graphs in the Gaussian plane for the complex valued function $G(t|a,b,c,d;\,G_0,G_1)$ with real argument $t$ (according to equation \ref{Gtcomplex}). The small circles at the intersection with the axis of the real part, i.e., at the zeros of the function $\mathtt{Im}[G(t)]$, provide the values of the sequence $\{G_n(a,b,c,d;G_0,G_1)\}$. This occurs for all integer values of the variable, i.e., for $t=n \in \N_0$. \emph{Upper left:} Oscillatory graph for $G(t| 0.7,1.4,0.3,0.5;\,-0.5,2.5)$, corresponding to $0<b/\alpha<1$.	\emph{Upper right:} Spiral graph with $G(t| -0.7,1.4,0.3,0.5;-0.5,2.5)$, corresponding to $b/\alpha>1$. \emph{Lower left:} Examples with integer parameter values, once for $G(t|2,3,2,2;-1,-2)$ (thin line, $b/\alpha>1=1$) and once with genuine Fibonacci number parameter values $G(t|1,1,0,0;0,1)$ (thick line, $0<b/\alpha<1$). In the former case, the amplitude remains constant ($B\approx0.417$ in equ. 4.3), in the latter case there is damped oscillatory behaviour. The integer Fibonacci numbers appear at the intersection of the graph with the real part axis, i.e., $F_n=G(t|t=n)$. \emph{Lower right:} Spiraling graph for $G(t|-1,1,0,0;0,1)$ (corresponding to $b/\alpha>1$), including \emph{alternating} Fibonacci numbers at the zeros of the imaginary part. \normalsize}
\end{figure*}

\section{Generalized Fibonacci spirals with real argument}

So far, the index $n$ has been an integer number producing a real valued sequence $\{G_n\}$. Replacing now the index $n$ by some real argument $t$, we have the complex valued function
\begin{equation}
G(t)=G(t|a,b,c,d;G_0,G_1)\,=\,A\,\alpha^t+B\,\beta^t+p\,d^t,\hspace{0.5cm}t \in \mathbb{R}_0^+ \label{Gt}
\end{equation}
with $A$ and $B$ still given by equations \eqref{A311} and \eqref{B311}, respectively. Indeed, applying Euler's formula, i.e., $(-1)^t=\exp(i\pi\,t)=\cos(\pi\,t)+i\,\sin(\pi\,t)$, and reminding $\alpha\beta=-b$, one may express
\begin{eqnarray}
G(t)&=& A\,\alpha^t+B\,(-b/\alpha)^t+pd^t \\ 
&=&\Big( A\,\alpha^t+B\,\cos(\pi t)(b/\alpha)^t+pd^t\Big)+\,\,i\,\,\Big( B\,\sin(\pi t)(b/\alpha)^t\Big)\label{Gtcomplex}\\
&=& \mathtt{Re}[G(t)]+i\,\,\mathtt{Im}[G(t)].
\end{eqnarray}
This is analogous to the treatment by Horadam \& Shannon (1988) \cite{horadam88} for Fibonacci and Lucas numbers or by Horadam (1988) \cite{horadam88b} for Jacobsthal and Pell numbers. For a similar, but more recent approach, see, e.g., \cite{chandra} and \cite{ozvatan}. Drawing the generalized Fibonacci function $G(t)$ in the Gaussian plane creates either oscillatory or spiral curves, depending whether the parameter $b/\alpha$ obeys $0<b/\alpha\le1$ or $b/\alpha>1$, respectively. The former criterium results in the gradual disapperance of the trigonometic terms, while the latter criterium guarantees for nonvanishing contributions of both the cosinus term and the sinus term, providing the circularity of the graph. At the zeros of $\mathtt{Im}[G(t)]$ the equality $G(t)=G_n$ holds (with $G_n$ as introduced in Section 2), because for all integer values of $t$ one has $\cos(\pi\,t)=\pm1$ and $\sin(\pi\,t)=0$. Hence, as a \textit{corollary}, 
\begin{equation}
G(t) = G_n\hspace{0.5cm}\forall t \in \N_0.
\end{equation}
Figure 4 illustrates this kinship and the general behaviour of the graphs for pairwise nearly equal sets of parameter values. See figure caption for some details. The oscillating "Fibonacci curve" originally shown in \cite{horadam88} is depicted, too (lower left panel, thick line with red circles for the Fibonacci numbers). We finally note that some spacial representation of $G(t)$ could be managed in a similar manner as in Section 3.4. 

\section{Conclusions}

We draw generalized Fibonacci spirals, based on analytic solutions of the recurrence relation $G_n=a\, G_{n-1}+b\, G_{n-2}+c\, d\,^n$. The underlying generalized Fibonacci numbers $G_n$ are equivalent to transformed Horadam numbers $H_n=G_n-pd^n$. Our inquiry is restricted to positive real initial values $G_0$ and $G_1$ and coefficients $a$, $b$, $c$, and $d$, that additionally satisfy the conditions $a^2+4b>0$ (restriction 1) and $d^2-a\,d-b = (d-\alpha)(d-\beta)\ne 0$ (restriction 2). Complex solutions for $G_n$ (for $a^2+4b<0$) or degenerate cases with $d=\alpha$ or $d=\beta$, as well as allowing for negative parameter values could be dealt with in a follow-up study. The principal coordinates used to draw the spirals correspond to finite sums of alternating even- or alternating odd-indexed terms $G_{n}$ and are given in closed-form. For this closed-form solution $\Gamma_n$ (equation \ref{Gn}), two proofs are given (see Appendix), one based on the Moivre-Binet-form of $G_n$ and a shorter one based on the transformed Horadam numbers $H_n=G_n-pd^n$ that adopts the substitution method suggested in the text. For rectangular spirals composed of straight line segments, the even-indexed and the odd-indexed directional corner points asymptotically lie on mutually orthogonal oblique lines. We calculate the points of intersection and show them in the case of inwinding spirals to coincide with the calculable point of convergence. In the case of outwinding spirals, an $n$-dependent quadruple of points of intersection may form. We illustrate the situation and provide approximate coordinate calculations. For arched spirals, interpolation between principal coordinates is performed by means of arcs of quarter-ellipses. A simple three-dimensional representation that visualizes the exponential input $c\, d\,^n$ is exhibited, too.  Other choices can be thought of, and extensions up to spiraling surfaces in 3D space seem attractive. The continuation of the discrete sequence $\{G_n\}$ to the complex-valued function $G(t)$ with real argument $t$$\in$$\R$ exhibits spiral graphs and oscillating curves in the Gaussian plane, thereby subsuming the values $G_n$ for $t$ $\in$ $\N_0$ as the zeros.

Besides, we retrieve within our framework the Shannon identity (i.e., a generalization of the Tagiuri product difference Fibonacci identity) and suggest the substitution method in order to find a variety of other identities and summations related to $G_n$. This may pose problems suitable for the problem's section of The Fibonacci Quarterly \cite{FQ}. The Shannon identity for Horadam numbers in particular could provide helpful in looking for more and higher-order product identities for $H_n$ and hence for $G_n$. This would be in the spirit of former inquiries like those of Melham \cite{melham03}, \cite{melham11} and many others (e.g., \cite{fair05}) in the case of Fibonacci and Horadam numbers. For that purpose, it could be advantageous to rely on the matrix representation of $G_n$ in terms of transformed Horadam numbers (as provided in the Appendix), and to adopt methods as outlined in, e.g., \cite{wadill}, \cite{johnson}, \cite{cerda13}. Finally, it is to hope that the generalized Fibonacci spirals prove good for some artistic, technical, natural, or even extra-terrestrial applications.

\section{Appendix}

\subsection{Matrix representation}

Because the numbers $G_n$ are but transformed Horadam numbers $H_n$ (equ. \ref{Hn}), the system matrix for recurrent (real valued) Horadam numbers 
\begin{equation}\label{Gmatrix1}
\mathbb{H}=\begin{pmatrix}
a & b  \\                                              
1 & 0  \\                                            
\end{pmatrix}
\end{equation}
---called R-matrix in \cite{wadill}--- does represent recurrence relation \eqref{RR}, too, according to
\begin{equation}\label{Gmatrix2}
\begin{pmatrix}
H_n\\                                H_{n-1}\\                                            
\end{pmatrix}=\begin{pmatrix}
G_n-pd^n\\                                G_{n-1}-pd^{n-1}\\                                            
\end{pmatrix}=\mathbb{H}\,\begin{pmatrix}
G_{n-1}-pd^{n-1}\\                                G_{n-2}-pd^{n-2}\\                                            
\end{pmatrix}
=\mathbb{H}^{n-1}\,\begin{pmatrix}
G_1-pd\\                                G_0-p\\                                            
\end{pmatrix}
=\mathbb{H}^{n-1}\,\begin{pmatrix}
H_1\\                                H_0\\                                            
\end{pmatrix}.
\end{equation}
Its eigenvalues are $\alpha$ and $\beta$ as given by equation \eqref{lambda12} and its eigenvectors are $\begin{pmatrix}\alpha\\ 1 \end{pmatrix}$ and $\begin{pmatrix} \beta\\ 1 \end{pmatrix}$. Defining matrices $D=\begin{pmatrix}\beta& 0\\0 & \alpha \end{pmatrix}$ and $T=\begin{pmatrix}\beta& \alpha\\1 & 1 \end{pmatrix}$, with $D^n=\begin{pmatrix}\beta^n& 0\\0 & \alpha^n \end{pmatrix}$ and  $T^{-1}=\frac{1}{\alpha-\beta} \begin{pmatrix}-1& \alpha\\1 & -\beta \end{pmatrix}$, allows for diagonalization according to $\mathbb{H}=TDT^{-1}$. Hence, $\mathbb{H}^{n}=TDT^{-1}TDT^{-1}\dots TDT^{-1}=TD^{n}T^{-1}$ or
\begin{eqnarray}
\mathbb{H}^{n}&=&\frac{1}{\alpha-\beta}\begin{pmatrix}
\alpha^{n+1}-\beta^{n+1}& -\beta\alpha^{n+1}+\alpha\beta^{n+1}\\
\alpha^{n}-\beta^{n}&-\beta\alpha^{n}+\alpha\beta^{n}\end{pmatrix}
\label{Gmatrix3}\\
& = & \frac{1}{\alpha-\beta}\begin{pmatrix}
\alpha^{n+1}-\beta^{n+1}& b\,\big(\alpha^{n}-\beta^{n}\big)\\
\alpha^{n}-\beta^{n}& b\,\big(\alpha^{n-1}-\beta^{n-1}\big)\end{pmatrix}
\label{Gmatrix4}\\
&=&\begin{pmatrix}h_{n+1}& b\,h_n\\ h_n& b\,h_{n-1} \end{pmatrix},
\label{Gmatrix5}
\end{eqnarray}
where $h_n=G_n(a,b,0,0;0,1)=(\alpha^{n}-\beta^{n})/(\alpha-\beta)$ are the conjoined \textit{fundamental Horadam numbers} (i..e., with initial values 0 and 1). Equation \eqref{Gmatrix2} can easily be recast by means of equation \eqref{Gmatrix3}. For the Fibonacci numbers $F_n=G_n(1,1,0,0;0,1)$, for example, the last equation readily reduces to the well-known matrix 
\begin{equation}
\mathbb{F}^{n}=
\begin{pmatrix}F_{n+1}& F_n\\ F_n& F_{n-1} \end{pmatrix}.
\end{equation} 
Matrix representations are frequently used in the literature to produce summation identities and to establish other properties in relation to recurrent sequences (e.g., recently, \cite{catalani}, \cite{johnson}, \cite{koshy}). For example, the first component in equation \eqref{Gmatrix2} implies the \textit{decomposition}
\begin{equation}
G_n = h_n\,G_1+bh_{n-1}\,G_0+p\big( d^n-dh_n-bh_{n-1}\big). 
\end{equation} 
We do not pursuit the matrix method any further here.

\subsection{Some partial sums related to $\{G_n\}$} 

The main ingredients for the derivations of the summation formulae given below and in the next subsection are repeatedly the recurrence relation \eqref{RR}, the Moivre-Binet-type solution formula \eqref{Gngen}ff, and the partial sum formula for geometric progression with a factor $d\in \mathbb{R}$, i.e.,
\begin{equation}\label{Sumnd}
S_n(d) \equiv \sum_{k=0}^{n} d^{k}=
\begin{cases}
\,\frac{d^{n+1}-1}{d-1}\hspace{0.42cm}(d\ne 1)\\
\,\,\,n+1\hspace{0.52cm}(d=1)
\end{cases}
\end{equation} 

Proceeding either as exemplified in Horadam (1965) \cite{horadam65}, section 3, by repeated use of the recurrence relation or by insertion of the Moivre-Binet-type solution for $G_n$ one straightforwardly arrives at 
\begin{eqnarray}
\sum_{k=0}^{n} \,G_{k} 
&=& \frac{1}{a+b-1} \Bigg( (a-1)G_0-G_1+b\,G_n+G_{n+1}
-cd^2\,S_{n-1}(d)\Bigg) \label{sumGk1}\\ 
&=& \frac{1}{a+b-1} \Bigg( (a-1)(G_0-G_{n+1})-G_1+G_{n+2}
-cd^2\,S_{n}(d)\Bigg) \label{sumGk2}\\
&=& \frac{(a-1)H_0-H_1+b\,H_n+H_{n+1}}{a+b-1}
+p S_{n}(d). \label{sumGk3}
\end{eqnarray}
Herein $H_n=G_n-pd^n$ (equ. \ref{Hn}). The three alternative formulations are given for ease of comparability with results of different approaches. For example, replacing $n\rightarrow n-1$ in relation \eqref{sumGk2} one reproduces proposition 3$\,$i) in Phadte \& Valaulikar (2016, setting their $A=cd^2$) \cite{phadte16}. For $p=0$ one basically recovers the formula originally presented in \cite{horadam65}. Therefore, adopting the substitution method suggested in Section 2.2 and directly starting with identity \eqref{sumGk3} (with $p=0$) readily provides some sum formula for $G_n-pd^n$ and hence for $G_n$.  

Similarly, the corresponding result for cumulated differences becomes
\begin{eqnarray}
\sum_{k=0}^{n} (-1)^k\,G_{k} 
&=& \frac{1}{a-b+1} \bigg( (a+1)G_0-G_1+(-1)^n \Big(  G_{n+1}-b\,G_n\Big)+cd^2\,S_{n-1}(-d)\bigg) \label{sumGk4}\nonumber\\
& & \\
&=& \frac{1}{a-b+1} \bigg( (a+1)\Big(G_0+(-1)^nG_{n+1}\Big)-G_1+(-1)^{n+1}G_{n+2} \nonumber\\& &\hspace{8.5cm}+cd^2 \,S_{n}(-d)\bigg)\label{sumGk5}\\
&=&\frac{(a+1)H_0-H_1+(-1)^n\Big(H_{n+1}-b\,H_n\Big)}{a-b+1}
+p S_{n}(-d). \label{sumGk6}
\end{eqnarray}
For some more partial sums of linear or quadratic terms regarding sequence ${G_n}$, see \cite{phadte16}. The distinct summation relation given by our equations \eqref{SumAlt}-\eqref{Gamman} is, however, original work by ours and proven below.

\subsection{Proof for $\Gamma_{n}$}

In Section 6.3.1, we prove the sum formula \eqref{SumAlt}f for alternating even-indexed or alternating odd-indexed generalized Fibonacci numbers by means of the Moivre-Binet-form solution of $G_n$. In Section 6.3.2, in a second proof of equation \eqref{SumAlt}f, we presume availability of the sum formula for the special case of Horadam numbers (equation \ref{GammaHora}) and proceed by means of the substitution method suggested in Section 2. 

\subsubsection{Relying on the  Moivre-Binet-form solution.}  For a proof of equation \eqref{SumAlt}f we distinguish in advance the summations both according to $n$ being even or odd and with respect to the further constraint $n\,\mathtt{mod}\,4 =\,0$ or 2 (if $n$ even) or $n\,\mathtt{mod}\,4 =\,$1 or 3 (if $n$ odd). This latter distinction seems necessary due to the changing occurence of equal or unequal numbers of summands with positive or with negative signs (cf. Table 1).

The formula for the case $n$ even and $n\,\mathtt{mod}\,4=0$ (i.e., $n=4,8,12,\dots$) derives in detail as follows (with notations $\cdots$ and $(..)$ abbreviating some similar treatments or expressions, respectively):
\begin{equation}
\sum_{k=0}^{n/2} (-1)^{k}\,G_{2k} =
\sum_{k=0}^{\frac{n}{4}} \,G_{4k} - 
\sum_{k=0}^{\frac{n-4}{4}} \,G_{4k+2}\hspace{7.5cm}
\end{equation}
\begin{eqnarray}
&=&A\sum_{k=0}^{\frac{n}{4}}(\alpha^4)^k +B\sum_{k=0}^{\frac{n}{4}}(\beta^4)^k+p\sum_{k=0}^{\frac{n}{4}}(d^4)^k-A\alpha^2\sum_{k=0}^{\frac{n-4}{4}}(\alpha^4)^k -B\beta^2\sum_{k=0}^{\frac{n-4}{4}}(\beta^4)^k -pd^2\sum_{k=0}^{\frac{n-4}{4}}(d^4)^k   
\nonumber\\
&=& A\frac{\alpha^{n+4}-1}{\alpha^4-1} +B\,\dots+p\,\dots 
 -A\alpha^2\frac{\alpha^{n}-1}{\alpha^4-1} -B\,\dots-p\,\dots 
\nonumber\\
&=&A\frac{\alpha^{n+2}(\alpha^{2}-1)+(\alpha^{2}-1)}{(\alpha^2-1)(\alpha^2+1)} +B\,\dots+p\,\dots \nonumber \\
&=&A\frac{\alpha^{n+2}+1}{\alpha^2+1} +B\frac{\beta^{n+2}+1}{\beta^2+1}
+p\frac{d^{n+2}+1}{d^2+1}\\
&=&A\frac{(\alpha^{n+2}+1)(1+\beta^2)(d^2+1)} {(\alpha^2+1)(\beta^2+1)(d^2+1)} +B\,\dots+p\,\dots 
\nonumber\\
&=&\frac{A \bigg( \alpha^{n+2} +(\alpha\beta)^2\alpha^n+ 1+\beta^2\frac{\alpha^2}{\alpha^2} \bigg)(d^2+1)}{(..)(..)(..)}
+\frac{B \bigg( \beta^{n+2} +(\alpha\beta)^2\beta^n+ 1+\alpha^2\frac{\beta^2}{\beta^2} \bigg)(d^2+1)}{(..)(..)(..)}
\nonumber \\
& &+\frac{p \bigg( d^{n+2} +(\alpha\beta)^2d^n+ 1+\alpha^2\beta^2\frac{1}{d^2} \bigg)(d^2+1)}{(..)(..)(..)}
-\frac{p \bigg( d^{n+2} +(\alpha\beta)^2d^n+ 1+\alpha^2\beta^2\frac{1}{d^2} \bigg)(d^2+1)}{(..)(..)(..)}
\nonumber 
\end{eqnarray}
\begin{eqnarray}
& &+\frac{p \bigg( d^{n+2} +1 \bigg)(1+\alpha^2)(1+\beta)^2}{(..)(..)(..)}
\,\,\,\, //\mathtt{equ.}\,\eqref{alphabeta1}, \eqref{alphabeta2},\eqref{alphabeta3}
\nonumber \\
&=&\frac{\bigg( G_{n+2}+b^2G_n+G_0+b^2G_{-2}\bigg)(d^2 +1)}{(..)(..)(..)}
-\,\frac{p(d^n+d^{-2})\bigg(d^4-(\alpha^2+\beta^2)d^2+(\alpha\beta)^2\bigg)}{(..)(..)(..)}
\nonumber \\
&=&\frac{G_{n+2}+b^2G_n+G_0+b^2G_{-2}}{a^2+(b+1)^2}
-\frac{p\bigg(d^4-(a^2+2b)d^2+b^2\bigg)(d^n+d^{-2})}{(a^2+(b+1)^2)(d^2+1)}\hspace{0.7cm} //\mathtt{equ.}\,\eqref{Gnpart}
\nonumber \\
&=& \frac{1}{a^2+(b+1)^2} \bigg\{  G_{n+2}+b^2 G_n  +G_0 +b^2G_{-2}-c\,\frac{d^2+ad-b}{d^2+1} \,
\Big( \,d^{n+2}+1\, \Big)\,\bigg\}.
\label{nmod4eq0} 
\end{eqnarray}
With the last step one observes the polynom division $(d^4-(a^2+2b)d^2+b^2)/(d^2-ad-b)=d^2+ad-b$. 

Omitting the details, in a very similar way as above one obtains for the case $n$ even and $n\,\mathtt{mod}\,4=2$ (i.e., $n=6,10,14,\dots$):
\begin{equation} 
\sum_{k=0}^{n/2} (-1)^{k}\,G_{2k} = 
\sum_{k=0}^{\frac{n-2}{4}} \,G_{4k} - 
\sum_{k=0}^{\frac{n-2}{4}} \,G_{4k+2}\hspace{7.5cm}
\end{equation}
\begin{eqnarray}
&=&-\,A\frac{\alpha^{n+2}-1}{\alpha^2-1} -\,B\frac{\beta^{n+2}-1}{\beta^2+1}
-\,p\frac{d^{n+2}-1}{d^2+1}\hspace{7.5cm} \\ 
&=& \frac{1}{a^2+(b+1)^2}\bigg\{-\Big( G_{n+2}+b^2 G_n \Big) +G_0 +b^2G_{-2}-c\,\frac{d^2+ad-b}{d^2+1}\Big( -d^{n+2}+1 \Big)\bigg\}.
\label{nmod4eq2} 
\end{eqnarray}
For the case $n$ odd and $n\,\mathtt{mod}\,4=1$ (i.e., $n=1,5,9,\dots$) one has:
\begin{equation} 
\sum_{k=0}^{(n-1)/2} (-1)^{k}\,G_{2k+1} = 
\sum_{k=0}^{\frac{n-1}{4}} \,G_{4k+1} - 
\sum_{k=0}^{\frac{n-5}{4}} \,G_{4k+3}\hspace{6.5cm}
\end{equation}
\begin{eqnarray}
&=&A\frac{\alpha^{n+2}+\alpha}{\alpha^2+1} +B\frac{\beta^{n+2}+\beta}{\beta^2+\beta}
+p\frac{d^{n+2}+d}{d^2+1}\hspace{7.5cm} \\ 
&=& \frac{1}{a^2+(b+1)^2} \bigg\{ G_{n+2}+b^2 G_n +G_1 +b^2G_{-1}-c\,\frac{d^2+ad-b}{d^2+1}
\Big( d^{n+2}+d \Big)\bigg\}.
\label{nmod4eq1} 
\end{eqnarray}
Finally, the case $n$ odd and $n\,\mathtt{mod}\,4=3$ (i.e., $n=7,11,15,\dots$) gives
\begin{equation} 
\sum_{k=0}^{(n-1)/2} (-1)^{k}\,G_{2k+1} = 
\sum_{k=0}^{\frac{n-3}{4}} \,G_{4k+1} - 
\sum_{k=0}^{\frac{n-3}{4}} \,G_{4k+3}\hspace{6.5cm}
\end{equation}
\begin{eqnarray}
&=&-\,A\frac{\alpha^{n+2}-\alpha}{\alpha^2+1} -\,B\frac{\beta^{n+2}-\beta}{\beta^2+1}
-\,p\frac{d^{n+2}-d}{d^2+1}\hspace{7.5cm} \\ 
&=& \frac{1}{a^2+(b+1)^2} \bigg\{ -\Big( G_{n+2}+b^2 G_n \Big) +G_1 +b^2G_{-1}-c\,\frac{d^2+ad-b}{d^2+1} \,
\Big( -d^{n+2}+d \Big) \bigg\}.
\label{nmod4eq3} 
\end{eqnarray} 
Drawing a comparison of the results \eqref{nmod4eq0}, \eqref{nmod4eq2}, \eqref{nmod4eq1}, and \eqref{nmod4eq3} allows to unify for all $n$ according to equations \eqref{SumAlt} and \eqref{Gamman}.$\hspace{8.4cm}\square$

\subsubsection{Applying the substitution method.} To start with, only even-indexed terms are considered.  Inserting the transformation \eqref{Hn} into the expression for the series and into equation \eqref{GammaHora} gives
\begin{eqnarray}
\Gamma_{n}^{\mathtt{Hora}}&=&\sum_{k=0}^{\frac{n}{2}} \,(-1)^k H_{2k} =\sum_{k=0}^{\frac{n}{2}} (-1)^k G_{2k}-p\sum_{k=0}^{\frac{n}{2}}(-1)^k d^{2k} \nonumber\\
&=& \Gamma_n-p\sum_{k=0}^{\frac{n}{2}}(-d^2)^k =\,\,\,\Gamma_n-p\frac{(-1)^{n/2}d^{n+2}+1}{d^2+1}\label{627}
\end{eqnarray}
and
\begin{equation}
(a^2+(b+1)^2\,\,\Gamma_{n}^{\mathtt{Hora}} =
 (-1)^{n/2} \Big( H_{n+2}+\,b^2H_{n} \Big) +H_{0}+b^2 H_{-2} \hspace{2cm}\nonumber\\
\end{equation}
\begin{eqnarray}
&=& (-1)^{n/2} \Big( G_{n+2}-pd^{n+2}+b^2G_{n}-pb^2d^n \Big) +G_{0}-p+b^2 G_{-2}-pb^2d^{-2} \nonumber\\
&=&(-1)^{n/2} \Big( G_{n+2}+b^2G_{n}\Big) +G_{0} +G_{-2} 
-p\frac{d^2+b^2}{d^2} \Big( (-1)^{n/2}d^{n+2} +1  \Big).\label{628}
\end{eqnarray}
Solving equation \eqref{627} for $\Gamma_n$, inserting \eqref{628}, replacing $p$ according to \eqref{Gnpart}, and making use of the relation $(d^2+ad-b)(d^2-ad-b)=d^4-(a^2+2b)d^2+b^2$ (as in the proof above), $\Gamma_n$ as given in equation \eqref{Gamman} is easily recovered. To end with, a very similar derivation for the summation of odd-indexed terms can be accomplished, and the proof is complete.

%

\medskip

\noindent MSC2010: 11B37, 11B39, 14H50, 39A06, 65Q30

\end{document}